\documentclass[a4paper,12pt]{article}
\usepackage{graphicx}%
\usepackage{multirow}%
\usepackage{amsmath,amssymb,amsfonts}%
\usepackage{amsthm}%
\usepackage{mathrsfs}%
\usepackage[title]{appendix}%
\usepackage{xcolor}%
\usepackage{textcomp}%
\usepackage{manyfoot}%
\usepackage{booktabs}%
\usepackage{algorithm}%
\usepackage{algorithmicx}%
\usepackage{algpseudocode}%
\usepackage{listings}%
\usepackage{mathtools} 
\usepackage{accents} 
\usepackage{bm} 
\usepackage{enumitem} 
\usepackage[noabbrev]{cleveref} 
\usepackage[normalem]{ulem} 

\newcommand{\ud}{\, \mathrm{d}}
\newcommand{\R}{\mathbb{R}}
\newcommand{\inner}[2]{\langle #1, #2 \rangle}
\newcommand{\mutx}{\mu_{(t,x)}}

\newcommand{\interior}[1]{%
  {\kern0pt#1}^{\mathrm{o}}%
}

\definecolor{dkgreen}{rgb}{0,0.4,0}
\definecolor{dkblue}{rgb}{0,0,0.8}

\DeclareMathOperator*{\esssup}{ess\,sup}

\newtheorem{theorem}{Theorem}
\newtheorem{proposition}{Proposition}
\newtheorem{lemma}{Lemma}

\newtheorem{remark}{Remark}%
\newtheorem*{remark*}{Remark}%

\newtheorem{definition}{Definition}%
\newtheorem{assumption}{Assumption}%

\begin{document}
	
	\title{Concentration of measure-valued solutions\\ for semilinear parabolic equations}
	
	\author{Charlie Lebarbé$^{1,2}$, \'{E}milien Flayac$^1$, Michel Fournié$^1$,\\ Didier Henrion$^{2,3}$, Milan Korda$^{2,3}$}
	
\footnotetext[1]{Fédération ENAC ISAE-SUPAERO ONERA, Université de Toulouse, 10 Av. Marc Pélegrin, 31055 Toulouse, France}

\footnotetext[2]{LAAS-CNRS, Université de Toulouse, 7 Av. Col. Roche, 31400 Toulouse, France}

\footnotetext[3]{Faculty of Electrical Engineering, Czech Technical University in Prague, Technick\'{a} 2, 16626 Prague, Czechia.}
	
	\date{\today}
	\maketitle

\begin{abstract}The moment-sum-of-squares  hierarchy provides a powerful framework for solving non-convex optimal control problems by constructing a sequence of convex semidefinite relaxations. However, when extending these methods to nonlinear partial differential equations (PDEs), a fundamental challenge is the potential existence of a relaxation gap, where the solution to the linear measure formulation using occupation measures fails to correspond to a classical physical solution of the original PDE. In this paper, we prove the absence of a relaxation gap for scalar semilinear parabolic PDEs of the reaction-diffusion type. We do so by showing that each solution to the linear measure equation gives rise to an energy measure-valued (emv) solution in the space of Young measures satisfying suitable energy identities. We then prove that any such emv solution concentrates on the solution to the nonlinear PDE, provided the latter exists and is unique.  
To the best of our knowledge, this is the first concentration result of this kind for measure-valued solutions of reaction-diffusion PDEs.
	\end{abstract}

\section{Introduction}\label{sec1}

The moment-sum-of-squares or Lasserre hierarchy, introduced by Lasserre in \cite{lasserre2001}, provides a powerful framework for solving non-convex polynomial optimization problems to global optimality, see e.g. \cite{nie2023,theobald2024} and references therein. This method constructs a sequence of convex semidefinite programming relaxations of increasing size, whose solutions are guaranteed to converge to the true global optimum. The success of this approach led to its extension to the realm of optimal control for systems governed by polynomial ordinary differential equations (ODEs) \cite{lasserre2008nonlinear}, building upon foundational work on linear formulations of optimal control problems using Young measures and occupation measures \cite{young1969,lv1980,rubio1976extremal,vinter1993convex,hht1996}. More recently, these techniques have been further extended to tackle the significantly more challenging domain of nonlinear partial differential equations (PDEs), see \cite{khl2020,KHL-2018} and references therein.

A central challenge when applying these methods to nonlinear PDEs is the potential existence of a \emph{relaxation gap}. This gap represents a fundamental discrepancy between the solution of the original nonlinear, non-convex problem and the solution of its convex relaxation formulated in the space of measures. If a gap exists, the solution to the linear problem may not correspond to any classical solution of the original nonlinear problem, potentially representing a statistical average over multiple non-physical solutions \cite{ft2022,kr2022}. The presence of such a gap can undermine the convergence guarantees of numerical schemes based on the moment-SOS hierarchy.

Recent research has focused on identifying classes of problems where this relaxation gap can be proven to be zero.

For convex optimal control problems involving both ODEs and PDEs, it has been shown that there is no gap \cite{hkkr2024}. For non-convex problems, the situation is more delicate. For nonlinear ODEs, no gap is implied by the convexity of the minimized Lagrangian as a function of the admissible velocities (even if the problem as a whole is non-convex)~\cite{rubio1976extremal,vinter1993convex} or by geometric conditions~\cite{augier2025sufficient}. In general, for ODEs, a relaxation gap is a rare and a rather pathological phenomenon~\cite{palladino2014minimizers,augier2026young}. For nonlinear PDEs, the landscape is much richer. Convexity of the minimized Lagrangian as a function of the admissible velocities is sufficient for the absence of a relaxation gap only in co-dimension one (i.e., the unknown function is a scalar function of possibly many variables) but counterexamples exist in higher codimensions; these results were recently proved in \cite{kr2022}. Besides these general results, one has to argue on a case-by-case basis, tailoring the arguments to the specific problem at hand.

For scalar hyperbolic conservation laws, which are inherently non-convex, it has been shown that the relaxation gap is closed by enforcing entropy inequalities, which are sufficient to single out the unique, physically meaningful solution \cite{burgers,cardoen2024}. Similarly, for the specific hyperbolic-parabolic cross-diffusion system studied in \cite{mvsol_hyperbolic_parabolic_2025}, the authors prove that any measure-valued solution satisfying the relevant Shannon/Rao entropy inequalities concentrates on the classical solution of the PDE system, whenever such a classical solution exists. Other approaches have used SOS techniques to compute bounds for nonlinear PDEs, particularly in fluid dynamics, though without addressing the questions of a relaxation gap and convergence guarantees \cite{bc2006,gc2012,cghp2014,vap2016,gf2019,fgc2022}. The optimal control of linear PDEs has also been addressed, but again, without guarantees of convergence to the global optimum \cite{magron2020}. A distinct line of work has tackled the relaxation gap problem by considering occupation measures supported on infinite-dimensional functional spaces. In this context, no relaxation gap can be ensured for a very broad class of evolution equations with dissipative operators \cite{chhatoi2024optimizing}. Galerkin approximations are then used to formulate tractable finite-dimensional moment problems \cite{flmw2020,hikv2023,hr2025}, but the efficiency of these methods in optimization and optimal control remains to be established. 

This work focuses on \emph{semilinear parabolic PDEs} of the reaction-diffusion type. Our main contribution is to prove that for this important class of equations, \emph{there is no relaxation gap} in the linear formulation on occupation measures from \cite{KHL-2018} working with measures that have finite-dimensional support on subsets of the Euclidean space. We do so by proving that the occupation measure relaxation of the nonlinear PDE has a unique solution concentrated on the solution to the nonlinear PDE, provided the latter exists and is unique. Compared to the formulation from \cite{KHL-2018}, we add one more affine constraint that forces the concentration of the time derivative of the measure-valued solution on the time derivative of the solution to the PDE.

To the best of our knowledge, it is the first result of concentration of measure-valued solutions for reaction-diffusion PDEs. This equivalence provides a sound theoretical foundation for applying the moment-SOS hierarchy to reaction-diffusion systems, guaranteeing that the sequence of solutions obtained from the semidefinite relaxations will converge to the moments of the true classical solution. In this paper, we focus on the fundamental problem of formulating the nonlinear PDE  as a linear measure equation, i.e. there is no parametrization of the set of initial conditions, and no external control input. The boundary conditions of the linear measure equation are concentrated on the boundary conditions of the nonlinear PDE. This is the first building block toward applications that include control and/or parametric boundary data.

The remainder of the paper is organized as follows. Section \ref{sec:notations} introduces the notations and functional spaces used throughout our analysis. In Section \ref{sec:problem}, we formulate the semilinear reaction-diffusion problem, detail the assumptions regarding the nonlinear reaction term, and establish the existence, uniqueness, and energy identities of its weak solutions. Section \ref{section:measure_valued_solutions} introduces the measure-valued framework: we define Young measures with finite moments and appropriate integration-by-parts formulas, leading to the definition of energy measure-valued (emv) solutions. This section culminates in the statement of our main result (Theorem \ref{theorem:concentration_emv_solutions}), which establishes the concentration of emv solutions. Section \ref{sec:proof} is entirely dedicated to the rigorous proof of this main theorem, utilizing a squared-error density argument to demonstrate that the measure-valued solutions concentrate on the classical weak solution. Section \ref{sec:occupation} details the connection between emv solutions and the occupation-measure formulation. Finally, Section \ref{sec:conclusion} provides concluding remarks. Deferred technical proofs of the supporting propositions are collected in the appendices \ref{appendix:proof_properties_sobolev}, \ref{appendix:proof_uniqueness_weak_solutions} and \ref{appendix:proof_energy_identity_for_weak_solutions}.

\section{Notations}\label{sec:notations}

Let $X\subset\mathbb{R}^n$ for a positive integer $n$. We write $C(X)$ for the space of continuous functions from $X$ to $\R$, and $C_b(X)$ for the subspace of functions in $C(X)$ that are bounded. For $k\in\mathbb{N}\cup\{\infty\}$, $C^k(X)$ denotes the space of $k$-times continuously differentiable functions, while $C_c^k(X)$ denotes its subspace of compactly supported functions. We denote by $\mathcal{M}(X)$ the space of finite signed Borel measures on $X$, by $\mathcal{M}_+(X)$ the cone of finite nonnegative Borel measures, and by $\mathcal{P}(X)$ the set of probability measures. Given $\mu\in\mathcal{M}_+(X)$ and $1\le p\le\infty$, $L^p(X,\mu)$ denotes the usual Lebesgue space. In particular, we write $L^p(X)=L^p(X,\mathcal{L}^n)$ for the Lebesgue measure $\mathcal{L}^n$ restricted to $X$. For integers $k\geq1$, we denote by $W^{k,p}(X)$ the Sobolev space of functions $y\in L^p(X)$ whose weak derivatives $\partial_\alpha y$ belong to $L^p(X)$ for each multiindex $\left\lvert\alpha\right\rvert\le k$. We write $H^k(X)=W^{k,2}(X)$ and we denote by $W_0^{k,p}(X)$ the closure of $C^\infty_c(X)$ in $W^{k,p}(X)$. We denote by $H^{-1}(X)$ the dual space of $H^1_0(X)$.

Let $T>0$ and let $B$ be a real Banach space with norm $\left\lVert \cdot\right\rVert_B$. We denote by $C([0,T];B)$ the space of continuous functions $y\colon[0,T]\to B,\, t \mapsto y(t,\cdot)$, where $y(t,\cdot) \in B$ represents the slice of the map $y$ at time $t$. We denote by $L^{p}(0,T;B)$, $1\le p\le\infty$, the Bochner space of (strongly) measurable functions $y\colon[0,T]\to B$ such that the map $t\mapsto \left\lVert y(t,\cdot)\right\rVert_{B}$ belongs to $L^{p}((0,T))$. 
For $1\leq p<\infty$ and when no confusion arises, we
identify $L^p(0,T;L^p(X))$ with the usual Lebesgue space $L^p((0,T)\times X)$. Similarly, we denote by $W^{1,p}(0,T;B)$ the Bochner-Sobolev space of functions $y\in L^p(0,T;B)$ such that $\partial_t y \in L^p(0,T;B)$. 

\section{Problem formulation}\label{sec:problem}

Let $T>0$ be a finite time horizon and let $\Omega$ be an open bounded domain of $\R^n$ ($n\geq1$) with locally Lipschitz boundary. We consider the following scalar semilinear parabolic reaction-diffusion PDE
\begin{equation}
\label{eq:semilinear_PDE}
\left\{
\begin{alignedat}{2}
\partial_t y(t,x) - \Delta_x y(t,x) &= f(t,x,y(t,x)), &\quad& t \in (0,T], \, x \in \Omega,\\
y(t,x) &= 0, &\quad& t\in[0,T], \, x \in \partial \Omega,\\
y(0,x) &= y_0(x), &\quad& x\in\Omega,
\end{alignedat}
\right.
\end{equation}
where $y$ is the solution (in a sense and in a functional space to be specified below) and $y_0 \in W^{1,r}_0(\Omega)$ is a prescribed initial condition. Throughout this paper, $r$ is a fixed exponent satisfying
\begin{equation}
\label{eq:definition_r}
    r \geq 2 \quad \text{and} \quad r>n.
\end{equation}
In this particular setting, the Sobolev embedding theorem (see \cite[Theorem 4.12]{adams_sobolev_2003}) yields the continuous embedding
\begin{equation}
\label{eq:continuous_embedding_W_1r}
W^{1,r}(\Omega)\hookrightarrow C(\overline{\Omega}).
\end{equation}
Furthermore, $W^{1,r}(\Omega)$ is a Banach algebra under pointwise multiplication, i.e. $uv\in W^{1,r}(\Omega)$ for all $u,v\in W^{1,r}(\Omega)$, and the product map is continuous; see \cite[Theorem 4.39]{adams_sobolev_2003}. The same conclusions also hold for $W^{1,r}_0(\Omega)$.
\begin{remark}
    Because $\Omega$ is bounded with locally Lipschitz boundary, the geometric assumptions appearing in the above results are automatically satisfied. Indeed, in the sense of \cite{adams_sobolev_2003}, such a domain satisfies both the \emph{strong local Lipschitz condition} and the \emph{cone condition}.
\end{remark}
The following proposition gathers two basic product estimates deduced from the Sobolev embedding theorem.
\begin{proposition}
\label{prop:properties_sobolev}
Let $r\in\R$ be such that $r\geq2$ and $r>n$.
\begin{enumerate}[label={(\roman*)}, ref={\theproposition (\roman*)}]
\item For any $(u,v) \in W^{1,r}(\Omega)\times H^1(\Omega)$, one has $uv \in H^1(\Omega)$. \label{prop:properties_sobolev_W1r}
\item For any $1/2 \leq \alpha \leq 1$ and any $(u,v) \in L^{\alpha r}(\Omega)\times H^1(\Omega)$, one has $uv \in L^{2\alpha}(\Omega)$. \label{prop:properties_sobolev_Lr2}
\end{enumerate}
\end{proposition}
\begin{proof}
    See Appendix \ref{appendix:proof_properties_sobolev}.
\end{proof}
We make the following assumptions on the nonlinear term $f\colon [0,T]\times\Omega\times \R\to\R$ in the PDE \eqref{eq:semilinear_PDE}.
\begin{assumption}
\label{assump:locally_bounded_f}
    For almost all $(t,x) \in [0,T]\times \Omega$,\, $y\mapsto f(t,x,y)$ is continuous on $\R$. Moreover, the function $f$ is essentially bounded in $(t,x)$  and locally bounded in $y$, i.e., for any compact set $K\in\R$, there exists $M_K\geq0$ such that
    \begin{align}
    \label{eq:boundedness_f}
       \esssup_{(t,x)\in [0,T]\times \Omega} \,\, \sup_{y\in K}
\left\lvert f(t,x,y) \right\rvert \le M_K,
    \end{align}
    where $\esssup$ denotes the essential supremum.
\end{assumption}
\begin{assumption}
\label{assump:OSL_f}
 The function $f$ is locally one-sided Lipschitz in the variable $y$ essentially in $(t,x)$, i.e., for any compact set $K\subset \R$, there exists $L_K\geq0$ such that for almost all $(t,x)\in[0,T]\times\Omega$ and any $(u,v)\in K^2$,
\begin{equation}
\label{eq:OSL_f}
 (u-v)(f(t,x,u)-f(t,x,v)) \leq L_K(u-v)^2.
\end{equation}
\end{assumption}
\begin{remark}
Note that Assumption \ref{assump:locally_bounded_f} does not imply Assumption \ref{assump:OSL_f}. Take for instance $f(t,x,y) = \sqrt{\left\lvert y \right\rvert}$, which is continuous and satisfies \eqref{eq:boundedness_f}, but is not one-sided Lipschitz in the neighborhood of 0. Conversely, Assumption \ref{assump:OSL_f} does not imply Assumption \ref{assump:locally_bounded_f} since $f(t,x,y) = 1/(t-T/2)$ satisfies \eqref{eq:OSL_f} but not \eqref{eq:boundedness_f}.
\end{remark}
\begin{remark}
   Notice that Assumptions \ref{assump:locally_bounded_f} and \ref{assump:OSL_f} are not particularly restrictive and cover a large class of nonlinearities. For example, every nonlinearity of the form $f(t,x,y) = g(y)$, where $g \in C^1(\R)$, satisfies these conditions. This includes many reaction terms commonly arising in reaction-diffusion models, such as polynomial and exponential nonlinearities. The qualitative behavior of solutions then depends on the precise structure of $f$, the spatial dimension $n$, and the initial condition $y_0$, and may exhibit a wide range of phenomena, including global existence and boundedness, convergence to steady states, or, in some regimes, finite-time blow-up. We refer to \cite{quittner_superlinear_2007} for an extensive study of these phenomena.
\end{remark}

In this work, we are interested in weak solutions of the PDE \eqref{eq:semilinear_PDE} living in the following space of functions:
\begin{align*}
    X_r &\coloneq  H^1(0,T;W^{1,r}_0(\Omega))= \Bigl \{ y\in  L^2(0,T;W^{1,r}_0(\Omega)) \,\text{ s.t. }\, \partial_t y\in L^2(0,T;W^{1,r}_0(\Omega)) \Bigr \}.
\end{align*}
\begin{remark}
\label{remark:continuous_embedding_Xr}
We know from \cite[Theorem 2, 5.9.2]{evans_pde} that any $\tilde{y} \in X_r$ admits a continuous-in-time representative $\hat{y} \in C([0,T];W^{1,r}_0(\Omega))$. The continuous embedding $W^{1,r}_0(\Omega)\hookrightarrow C(\overline{\Omega})$ then shows that $\hat y$ may also be regarded as an element of $C([0,T];C(\overline{\Omega}))$, and since $[0,T]\times\overline\Omega$ is compact, we can make the isometric identification $C([0,T]; C(\overline{\Omega})) \cong C([0,T]\times\overline{\Omega})$. Therefore, each element of $X_r$ admits a unique continuous representative on $[0,T]\times\overline{\Omega}$, which is in particular bounded. In what follows, we identify elements of $X_r$ with their representatives in $C([0,T]\times\overline{\Omega})$.
\end{remark}
In what comes next, we shall frequently use the square of functions belonging to $X_r$. The following elementary stability property ensures that this operation is admissible in $X_r$ and that the usual product rules remain valid in both space and time.
\begin{proposition}
\label{prop:product_in_Xr}
    For any $y \in X_r$, one has $y^2 \in X_r$. In particular, the product rule in space holds in $(L^r(\Omega))^n$ for any $t \in [0,T]$
\begin{equation}
\label{eq:product_rule_space_Xr}
\nabla_x(y(t,\cdot)^2) = 2y(t,\cdot) \nabla_x y(t,\cdot),
\end{equation}
and the product rule in time holds in $L^2(0,T;W^{1,r}_0(\Omega))$
\begin{equation}
\label{eq:product_rule_time_Xr}
\partial_t(y^2) = 2y \partial_t y.
\end{equation}
\end{proposition}
\begin{proof}
Let $y \in X_r \subset C([0,T];W^{1,r}_0(\Omega))$ (see Remark \ref{remark:continuous_embedding_Xr}). Since $W^{1,r}_0(\Omega)$ is a Banach algebra, it follows that $y^2 \in C([0,T];W^{1,r}_0(\Omega))$. In particular, $y(t,\cdot)^2 \in W^{1,r}_0(\Omega)$ for every $t \in [0,T]$, and the product rule in space \eqref{eq:product_rule_space_Xr} holds.

It remains to prove the time regularity of $y^2$. But since $y \in C([0,T]; W^{1,r}_0(\Omega))$ and $\partial_t y \in L^2(0,T;W^{1,r}_0(\Omega))$, one can use again the fact that $W^{1,r}_0(\Omega)$ is a Banach algebra to show that $y\partial_t y \in L^2(0,T;W^{1,r}_0(\Omega))$. Therefore, $y^2 \in H^1(0,T;W^{1,r}_0(\Omega)) = X_r$ and the weak time derivative of $y^2$ is given by \eqref{eq:product_rule_time_Xr}.
\end{proof}
We now introduce the notion of weak solution that will be used in this work for the PDE \eqref{eq:semilinear_PDE}.
\begin{definition}[Weak solution]
\label{def:weak_solutions}
We say that $y^* \in X_r$ is a weak solution to the PDE \eqref{eq:semilinear_PDE} if it satisfies for any $v \in H^1_0(\Omega)$ and any $\varphi \in C^\infty_c((0,T))$,
\begin{subequations}
\begin{gather}
\int_0^T \int_\Omega \varphi(t) \Bigl[ v(x) \partial_ty^*(t,x) + \nabla_x v(x) \cdot \nabla_x y^*(t,x) - v(x) f(t,x,y^*(t,x)) \Bigr] \ud x \ud t = 0, \label{eq:weak_solution}\\
y^*(0,x) = y_0(x), \quad x \in \Omega. \label{eq:initial_condition_weak_solution}
\end{gather}
\end{subequations}
\end{definition}
\begin{remark}
Notice that \eqref{eq:weak_solution} is equivalent to requiring that
for any $v \in H^1_0(\Omega)$ and for almost all $t\in[0,T]$,
\begin{equation}
\label{eq:weak_solution_ae_t}
    \int_\Omega  v(x) \partial_ty^*(t,x) + \nabla_x v(x) \cdot \nabla_x y^*(t,x) - v(x) f(t,x,y^*(t,x)) \ud x = 0.
\end{equation}
Since $y^* \in X_r \subset C([0,T]\times\overline{\Omega})$, it follows that $y^*$ is bounded on $[0,T]\times\overline{\Omega}$. By Assumption~\ref{assump:locally_bounded_f} on $f$, this in turn implies that the mapping $(t,x)\mapsto f(t,x,y^*(t,x))$ belongs to $L^\infty((0,T)\times\Omega)$. Moreover, since $\partial_t y^*(t,\cdot)$ and $\nabla_x y^*(t,\cdot)$ belong to $L^r(\Omega)\subset L^2(\Omega)$ for almost all $t\in[0,T]$, every term appearing in \eqref{eq:weak_solution} and \eqref{eq:weak_solution_ae_t} is integrable for any $v \in H^1_0(\Omega)$. Consequently, both expressions are well defined.
\end{remark}
With this notion of weak solution, we establish the following uniqueness result.
\begin{proposition}[Uniqueness of weak solutions]
\label{prop:uniqueness_weak_solution}
The PDE \eqref{eq:semilinear_PDE} admits at most one weak solution $y^* \in X_r$.
\end{proposition}
\begin{proof}
See Appendix \ref{appendix:proof_uniqueness_weak_solutions}.
\end{proof}
The previous proposition ensures that any weak solution, if it exists, is uniquely determined. In order to work with a well-defined solution throughout the remainder of the paper, we  make the following standing assumption.
\begin{assumption}[Local-in-time existence of weak solutions]
\label{assumption:existence_weak_solution}
For the given time horizon $T>0$, the PDE \eqref{eq:semilinear_PDE} admits a weak solution $y^* \in X_r$.
\end{assumption}
 By Proposition \ref{prop:uniqueness_weak_solution}, the solution to \eqref{eq:semilinear_PDE} is unique and we denote it as $y^*$. Furthermore, since $y^* \in X_r$, Remark \ref{remark:continuous_embedding_Xr} implies that $y^* \in C([0,T]\times\overline{\Omega})$. Therefore, the range of $y^*$ is compact and we denote by $\bm{Y}$ any compact subset of $\R$ containing the range:
\begin{equation}
\label{eq:compact_Y}
    \bm{Y} \supset y^*([0,T]\times\overline{\Omega}).
\end{equation}
\begin{remark}
To the best of our knowledge, no general local-in-time existence theorem for weak solutions under minimal assumptions on the domain $\Omega$ and on the nonlinearity $f$ is available in a form directly applicable to the present setting. Nevertheless, related results can be found in more regular settings: for instance, \cite[Proposition 7.3.2]{local_in_time_existence_strong_solution} proves the existence of a local-in-time classical solution to \eqref{eq:semilinear_PDE} when $\Omega$ has $C^2$ boundary.
\end{remark}
In the next proposition, we derive two variational identities satisfied by the weak solution $y^*$ to the PDE \eqref{eq:semilinear_PDE}. These identities involve $(y^*)^2$ and will be of particular interest later in Section~\ref{section:measure_valued_solutions} where they are used to prove the concentration of measure-valued solutions on $y^*$.
\begin{proposition}[Energy identities for weak solution]
\label{prop:energy_identity_for_weak_solutions}
Suppose that $y^*\in X_r$ is a weak solution to the PDE \eqref{eq:semilinear_PDE}. Then, it satisfies the following energy identities for any $v\in H^1_0(\Omega)$ and any $\varphi \in C^\infty_c((0,T))$:
\begin{multline}
\label{eq:weak_formulation_energy}
    \int_0^T \!\int_\Omega \varphi(t) \Bigl[
        v(x)\partial_t\bigl((y^*)^2\bigr)(t,x)
        + \nabla_x v(x) \cdot \nabla_x\bigl((y^*)^2\bigr)(t,x)
        - 2v(x)y^*(t,x)f(t,x,y^*(t,x))
        \\
        + 2v(x)\left\lVert \nabla_x y^*(t,x) \right\rVert^2
    \Bigr] \ud x \ud t = 0,
\end{multline}
\begin{equation}
\label{eq:weak_formulation_dissipation}
    \int_0^T \!\int_\Omega
        \varphi(t) (\partial_t y^*(t,x))^2
        - \frac{1}{2}\varphi'(t) \left\lVert \nabla_x y^*(t,x) \right\rVert^2
        - \varphi(t)\partial_t y^*(t,x) f(t,x,y^*(t,x)) \ud x \ud t = 0.
\end{equation}
\end{proposition}
\begin{proof}
See Appendix \ref{appendix:proof_energy_identity_for_weak_solutions}.
\end{proof}
\begin{remark}
Notice that \eqref{eq:weak_formulation_energy} is equivalent to requiring that for any $v\in H^1_0(\Omega)$ and for almost all $t\in[0,T]$,
\begin{multline}
\label{eq:weak_formulation_energy_ae_t}
    \int_\Omega 
        v(x)\partial_t\bigl((y^*)^2\bigr)(t,x)
        + \nabla_x v(x) \cdot \nabla_x\bigl((y^*)^2\bigr)(t,x)
        - 2v(x)y^*(t,x)f(t,x,y^*(t,x))
        \\
        + 2v(x)\left\lVert \nabla_x y^*(t,x) \right\rVert^2
    \ud x = 0.
\end{multline}
\end{remark}
\begin{remark}
Let us observe that identities \eqref{eq:weak_formulation_energy} and \eqref{eq:weak_formulation_dissipation} are obtained by testing the weak formulation \eqref{eq:weak_solution} with $v^1_t(x) = v(x)y^*(t,x)$ and $v^2_t(x) = \partial_t y^*(t,x)$, respectively. More generally, one can formally derive a broader class of identities by
testing \eqref{eq:weak_solution} with functions of the form
\begin{equation*}
    v_t(x) = p_1(x,y^*(t,x)) + p_2(x,y^*(t,x))y^*(t,x) +  p_3(x,y^*(t,x))\partial_t y^*(t,x),
\end{equation*}
where $p_1,p_2,p_3\in C^1(\overline{\Omega}\times\bm{Y})$ and
$p_1(x,y)=0$ on $\partial\Omega$ for every $y\in\bm{Y}$. Although these additional energy identities are not needed to prove the concentration result in Theorem~\ref{theorem:concentration_emv_solutions}, they may provide additional constraints when considering finite-dimensional relaxations of the moment problem associated with measure-valued solutions of the PDE. For instance, when $p$ is restricted to polynomials up to some degree $d$, one recovers the formulation of \cite{KHL-2018}, supplemented by constraints involving second-order moments of the time derivative of the solution.
\end{remark}

\section{Measure-valued solutions}
\label{section:measure_valued_solutions}

In this section, building on the work of DiPerna \cite{DiPerna_mv_solutions}, we introduce a notion of measure-valued solutions that extends the notion of weak solution to the PDE \eqref{eq:semilinear_PDE}. In this setting, the unknown is represented by a time- and space-dependent Young measure, rather than by a single pointwise state.

Since our objective is to show that measure-valued solutions concentrate on the solution of the original PDE \eqref{eq:semilinear_PDE}, the class of admissible measure-valued solutions must contain, at least, the Dirac measure
\begin{equation*}
    \delta^* \coloneq \delta_{\left(y^*(t,x),\partial_t y^*(t,x),\nabla_x y^*(t,x)\right)}
\end{equation*}
generated by the weak solution $y^*$ and its time and space derivatives. Consequently, any admissibility condition imposed at the measure-valued level has to correspond to a property already satisfied by $y^*$ once it is embedded into the measure-valued framework.

The set of compatible conditions considered in this section consists of an integrability condition, integration-by-parts identities, and first- and second-order moment constraints, introduced respectively in Definitions~\ref{def:young_measures_finite_p_moments}, \ref{def:IBP_formula} and \ref{def:emv_solutions} below. The main result of this work, Theorem~\ref{theorem:concentration_emv_solutions}, shows that these constraints are sufficient to establish the concentration of measure-valued solutions on $y^*$ and its first order derivatives.

We begin by recalling the notion of Young measures, also referred to as slicing Young measures or parametrized measures; see, for instance, \cite[Section~4.3.2]{attouch_variational_2006}.
\begin{definition}[Young measure]
\label{def:Young_measures}
Let $N\geq1$ and $\bm{E}\subset\R^N$. A \emph{Young measure} on $\bm{E}$ parametrized by $[0,T]\times\Omega$ is a map $\mu \colon [0,T]\times\Omega \to \mathcal{P}(\bm{E}), \,(t,x) \mapsto \mu_{(t,x)},$ such that for all $g \in C_b(\bm{E})$, the map
\begin{equation*}
\displaystyle (t,x) \mapsto \inner{\mutx}{g} \coloneq \int_{\bm{E}} g(\xi_1,\ldots,\xi_N) \ud \mu_{(t,x)} (\xi), \qquad \xi=(\xi_1,\ldots,\xi_N),
\end{equation*}
is measurable. We denote by $\mathcal{Y}([0,T]\times\Omega;\bm{E})$ the set of Young measures on $\bm{E}$ parametrized by $[0,T]\times\Omega$.
\end{definition}
Observe that the boundedness assumption on the function $g$ in Definition \ref{def:Young_measures} ensures that the pairing $\inner{\mutx}{g}$ is always well defined. However, as will become apparent later in the paper, this requirement is too restrictive when one wishes to consider functions with polynomial growth (for instance the monomials $g(\xi)=\left\lvert\xi_i\right\rvert^p$) on an unbounded set $\bm{E}\subset\mathbb{R}^N$. In such cases, the quantity $\inner{\mutx}{g}$ need not be finite a priori, nor need the corresponding moments of the measure $\mutx$. This remark motivates the following definition.
\begin{definition}[Young measure with finite $p$-moments]
\label{def:young_measures_finite_p_moments}
Let $p\geq1$. A Young measure $\mu \in \mathcal{Y}([0,T]\times\Omega;\bm{E})$ has \emph{finite $p$-moments} if $$\left( (t,x) \mapsto \int_{\bm{E}} \left\lVert \xi \right\rVert^p \ud \mutx(\xi) \right) \in L^1((0,T)\times\Omega).$$ In particular, this implies that
\begin{equation}
\label{eq:finite_p_moments_definition}
    \int_0^T\int_{\Omega} \left( \int_{\bm{E}} \left\lVert \xi \right\rVert^p \ud \mutx(\xi) \right) \ud x \ud t < \infty.
\end{equation}
We denote by $\mathcal{Y}^p([0,T]\times\Omega;\bm{E})$ the set of Young measures with finite $p$-moments.
\end{definition}
This notion of Young measures with finite $p$-moments is standard and is closely related---in fact equivalent, see \cite[Theorem 1]{finite_p_moments_reference}---to the notion of $L^p$-Young measures introduced in \cite{Lp_young_measures}. The latter are generated by bounded sequences in $L^p$ spaces and capture the limiting oscillatory behavior of such sequences.

In the context of this work, the relevance of Young measures with finite $p$-moments is that the defining measurability property can be extended from $C_b(\bm{E})$ to continuous functions with at most polynomial growth of order $p$, i.e., the functions $g$ such that
\begin{equation*}
\left\lvert g(\xi) \right\rvert \leq C_g\left(1 + \left\lVert \xi \right\rVert^p \right), \quad \forall \xi \in \bm{E},
\end{equation*}
for some constant $C_g \geq 0$. Furthermore, let us note that the control of the higher-order moments of $\mutx$ automatically yields control of all lower-order moments. More precisely, any Young measure with finite $p$-moments also has finite $q$-moments for any $q \in [1,p]$.

The following proposition shows that these lower-order moments inherit the integrability property \eqref{eq:finite_p_moments_definition} with the expected corresponding integrability exponent.
\begin{proposition}
\label{prop:integrability_Young_measures}
Let $p\geq1$ and $\mu \in \mathcal{Y}^p([0,T]\times\Omega;\bm{E})$. Then, for any $k \in [1,p]$,
\begin{equation}
\inner{\mutx}{\left\lVert \xi \right\rVert^k} \in L^{p/k}((0,T)\times\Omega).
\end{equation}
\end{proposition}
\begin{proof}
Let $p\geq 1$, $\mu \in \mathcal{Y}^p([0,T]\times\Omega;\bm{E})$ and $k \in [1,p]$. Then,
\begin{equation*}
\int_0^T \int_\Omega \left\lvert \inner{\mutx}{\left\lVert \xi \right\rVert^k} \right\rvert^{p/k} \ud x \ud t = \int_0^T \int_\Omega \left ( \int_E \left\lVert \xi \right\rVert^k \ud \mutx(\xi) \right)^{p/k} \ud x \ud t.
\end{equation*}
Using Jensen's inequality applied to the function $x\mapsto x^{p/k}$ (convex on $[0,\infty)$ since $p \geq k$) gives
\begin{align*}
    \int_0^T \int_\Omega \left\lvert \inner{\mutx}{\left\lVert \xi \right\rVert^k} \right\rvert^{p/k} \ud x \ud t &\leq \int_0^T \int_\Omega \int_E \left(\left\lVert \xi \right\rVert^k\right)^{p/k} \ud \mutx(\xi) \ud x \ud t\\
    &= \int_0^T \int_\Omega \int_E \left\lVert \xi \right\rVert^p \ud \mutx(\xi) \ud x \ud t\\
    & < \infty \text{ since $\mu$ has finite $p$-moments}.
\end{align*}
Thus, one has $\inner{\mutx}{\left\lVert \xi \right\rVert^k} \in L^{p/k}((0,T)\times\Omega)$.
\end{proof}
In the remainder of this paper, we will consider Young measures on $\bm{E}= \bm{Y}\times\bm{Z}$, where $\bm{Y}\subset\R$ is the compact space introduced in \eqref{eq:compact_Y} which contains the values of the bounded weak solution $y^*$ of the PDE, and $\bm{Z} = \R\times\R^{n}$ represents the space containing the values of the time and space derivatives of the weak solution. Let $\mu\in\mathcal{Y}([0,T]\times\Omega;\bm{Y}\times\bm{Z})$. For any $g\in C_b(\bm{Y}\times\bm{Z})$, the integration of the measure $\mutx$ against $g$ writes 
\begin{equation}
\label{eq:integration_bracket_mu_g}
    \inner{\mutx}{g} \coloneq \int_{\bm{Y}\times\bm{Z}} g(y,z) \ud \mutx(y,z),
\end{equation}
where the integrated variables $y$ and $z = \left[ z_0  \enspace z_1 \enspace \cdots \enspace z_n\right]^T$ belong to $\bm{Y}$ and $\bm{Z}$, respectively. It will sometimes be useful to treat the variable $z_0$ (which represents the time derivative of $y$) and the variables $z_1,\ldots,z_n$ (which represent the space derivatives of $y$) separately; for this reason we introduce the notation $\overline{z} \coloneq \left[z_1 \enspace \cdots \enspace z_n\right]^T$ to represent the space part of the time-space gradient $z$.

Moreover, if $g$ is a vector-valued function, then the integral in \eqref{eq:integration_bracket_mu_g} is taken componentwise. For instance, if $g(y,z) = \overline{z}$, one should understand $\inner{\mutx}{g}$ in the following sense:
\begin{equation*}
    \inner{\mutx}{\overline{z}} = \left[ \int_{\bm{Y}\times\bm{Z}} z_1 \ud \mutx(y,z) \enspace \cdots \enspace \int_{\bm{Y}\times\bm{Z}} z_n \ud \mutx(y,z) \right]^T.
\end{equation*}

In the next definition, we formulate an integration-by-parts identity at the level of Young measures. This condition links the coordinate $y$, representing the solution, with the coordinates $z_i$, representing its time and space derivatives. In particular, it is satisfied by the Dirac measure $\delta^*$ generated by the weak solution $y^*$ and its first-order derivatives.
\begin{definition}[Integration-by-parts formula for Young measures]
\label{def:IBP_formula}
A Young measure $\mu \in \mathcal{Y}^1([0,T]\times\Omega;\bm{Y}\times\bm{Z})$ satisfies the \emph{integration-by-parts formula} if it verifies the following identities for all $\beta\in C^1(\bm{Y})$ and all $\psi \in C^\infty_c((0,T)\times\Omega)$:
\begin{subequations}
\begin{align}
\int_0^T \int_\Omega \partial_t \psi(t,x) \inner{\mutx}{\beta(y)} \ud x \ud t &= -\int_0^T \int_\Omega \psi(t,x) \inner{\mutx}{z_0\beta^{'}(y)}\ud x \ud t, \label{eq:IBP_in_time}\\
\int_0^T \int_\Omega \nabla_x \psi(t,x) \inner{\mutx}{\beta(y)}\ud x \ud t &= -\int_0^T \int_\Omega \psi(t,x) \inner{\mutx}{ \overline{z} \beta^{'}(y)}\ud x \ud t.\label{eq:IBP_in_space}
\end{align}
\end{subequations}
\end{definition}
The following proposition clarifies the role of the integration-by-parts formula. For Young measures with finite $p$-moments, it provides the appropriate weak differentiability structure between the variables $y$ and $z$.
\begin{proposition}
\label{prop:existence_of_weak_derivative_for_Young_measures}
Let $p\geq1$ and suppose that a Young measure $\mu \in \mathcal{Y}^p([0,T]\times\Omega;\bm{Y}\times\bm{Z})$ satisfies the integration-by-parts formula. Then, for any $\beta\in C^1(\bm{Y})$,
\begin{equation}
    \inner{\mutx}{\beta(y)} \in W^{1,p}((0,T)\times\Omega) \cap L^\infty((0,T)\times\Omega),
\end{equation}
with
\begin{subequations}
\begin{align}
    \partial_t \inner{\mutx}{\beta(y)} &= \inner{\mutx}{z_0\beta'(y)},\\
    \nabla_x \inner{\mutx}{\beta(y)} &= \inner{\mutx}{\overline{z}\beta'(y)}.
\end{align}
\end{subequations}
Moreover, if $p \geq r$ (where $r$ is the exponent introduced in \eqref{eq:definition_r}), then $\inner{\mutx}{\beta(y)}$ admits a continuous representative (still denoted $\inner{\mutx}{\beta(y)}$) in the following sense:
\begin{equation}
    \inner{\mutx}{\beta(y)} \in C([0,T];L^r(\Omega)) \cap L^2(0,T;C(\overline{\Omega})).
\end{equation}
\end{proposition}
\begin{proof}
Let $\mu \in \mathcal{Y}^p([0,T]\times\Omega;\bm{Y}\times\bm{Z})$ for a given $p \geq 1$ and suppose that $\mu$ satisfies the integration-by-parts formula. Let $\beta \in C^1(\bm{Y})$. In particular, since $\bm{Y}$ is compact, we have $\beta\in L^\infty(\bm{Y})$ and $\beta'\in L^\infty(\bm{Y})$. Using the fact that $\mutx$ is a probability measure, we obtain
\begin{equation*}
\left\lvert \inner{\mutx}{\beta(y)} \right\rvert \leq \int_{\bm{Y}\times\bm{Z}} \left\lvert \beta(y) \right\rvert \ud \mutx(y,z) \leq \left\lVert \beta \right\rVert_{L^\infty(\bm{Y})} \int_{\bm{Y}\times\bm{Z}}1\ud \mutx(y,z) = \left\lVert \beta \right\rVert_{L^\infty(\bm{Y})}. 
\end{equation*}
Thus, $\inner{\mutx}{\beta(y)} \in L^\infty((0,T)\times\Omega)$. Now, for any $i \in \{0,1,\ldots,n\}$, one gets
\begin{align*}
\int_0^T \int_\Omega \left\lvert \inner{\mutx}{z_i\beta'(y)} \right\rvert^p \ud x \ud t &\leq \int_0^T \int_\Omega \left( \int_{\bm{Y}\times\bm{Z}} \left\lvert z_i \beta'(y) \right\rvert \ud \mutx(y,z) \right)^p \ud x \ud t\\
&\leq \left\lVert \beta' \right\rVert_{L^\infty(\bm{Y})}^p \int_0^T \int_\Omega \left( \int_{\bm{Y}\times\bm{Z}} \left\lvert z_i \right\rvert \ud \mutx(y,z) \right)^p \ud x \ud t.
\end{align*}
We use Jensen's inequality to obtain
\begin{equation*}
\int_0^T \int_\Omega \left\lvert \inner{\mutx}{z_i\beta'(y)} \right\rvert^p \ud x \ud t \leq \left\lVert \beta' \right\rVert_{L^\infty(\bm{Y})}^p \int_0^T \int_\Omega \left( \int_{\bm{Y}\times\bm{Z}} \left\lvert z_i \right\rvert^p \ud \mutx(y,z) \right) \ud x \ud t.
\end{equation*}
Since $\mu$ has finite $p$-moments, one has
\begin{equation*}
\int_0^T \int_\Omega \left( \int_{\bm{Y}\times\bm{Z}} \left\lvert z_i \right\rvert^p \ud \mutx(y,z) \right) \ud x \ud t < \infty,
\end{equation*}
and thus
\begin{equation*}
\int_0^T \int_\Omega \left\lvert \inner{\mutx}{z_i\beta'(y)} \right\rvert^p \ud x \ud t < \infty.
\end{equation*}
Therefore, $\inner{\mutx}{z_i\beta'(y)} \in L^p((0,T)\times\Omega)$ for all $i \in \{0,1,\ldots,n\}$ . Since $\mu$ satisfies the integration-by-parts formula, it follows that $\inner{\mutx}{\beta(y)} \in W^{1,p}((0,T)\times\Omega)$ with $\partial_t \inner{\mutx}{\beta(y)} = \inner{\mutx}{z_0\beta'(y)}$ and $\nabla_x \inner{\mutx}{\beta(y)} = \inner{\mutx}{\overline{z}\beta'(y)}$.

Finally, let us observe that
\begin{equation}
\label{eq:subset_W_1p}
    W^{1,p}((0,T)\times\Omega) \subset W^{1,p}(0,T;L^p(\Omega)) \cap L^p(0,T;W^{1,p}(\Omega)).
\end{equation}
Indeed, if $u \in W^{1,p}((0,T)\times\Omega)$, then by definition $u, \partial_t u, \nabla_x u \in L^p((0,T)\times\Omega)$. Then, by Fubini's theorem, this means that $u, \partial_t u \in L^p(0,T;L^p(\Omega))$ and $u, \nabla_x u \in L^p(0,T;L^p(\Omega))$. Therefore, $u \in W^{1,p}(0,T;L^p(\Omega))$ and $u \in L^p(0,T;W^{1,p}(\Omega))$, which proves \eqref{eq:subset_W_1p}.\\
\noindent Assume that $p\geq r$. Therefore $p\geq 2$ since $r$ satisfies \eqref{eq:definition_r}, and it follows directly from \eqref{eq:subset_W_1p} that
\begin{equation*}
    W^{1,p}((0,T)\times\Omega) \subset H^1(0,T;L^r(\Omega)) \cap L^2(0,T;W^{1,r}(\Omega)).
\end{equation*}
Using the inclusion $H^1(0,T;L^r(\Omega)) \subset C([0,T];L^r(\Omega))$ (see \cite[Theorem 2, 5.9.2]{evans_pde}) and the continuous embedding \eqref{eq:continuous_embedding_W_1r}, one eventually proves that
\begin{equation*}
W^{1,p}((0,T)\times\Omega) \subset C([0,T];L^r(\Omega)) \cap L^2(0,T;C(\overline{\Omega})).
\end{equation*}
\end{proof}
The next definition introduces four moment quantities that will play a central role in the measure-valued formulation of the PDE. These quantities are motivated by the weak formulation \eqref{eq:weak_solution} and by the energy identity \eqref{eq:weak_formulation_energy}.
\begin{definition}[First-order moment, second-order moment and integral source terms]
Let $\mu \in \mathcal{Y}^r([0,T]\times\Omega;\bm{Y}\times\bm{Z})$, where $r$ is the exponent introduced in \eqref{eq:definition_r}. The first- and second-order moments of $\mutx$ with respect to the variable $y$, respectively denoted $m_1$ and $m_2$, are defined for almost all $(t,x)\in[0,T]\times\Omega$ by
\begin{align*}
    m_1(t,x) &\coloneq \inner{\mutx}{y},\\
    m_2(t,x) &\coloneq \inner{\mutx}{y^2}.
\end{align*}
Two integral source terms of interest $m_f$ and $\hat{m}_f$ are also defined for almost all $(t,x)\in[0,T]\times\Omega$ by
\begin{gather*}
    m_f(t,x) \coloneq \inner{\mutx}{f(t,x,y)},\\
    \hat{m}_f(t,x) \coloneq 2\inner{\mutx}{yf(t,x,y)} -2\inner{\mutx}{\left\lVert \overline{z} \right\rVert^2}.
\end{gather*}
\end{definition}
The following proposition specifies the function spaces to which the moment quantities introduced above belong.
\begin{proposition}
\label{prop:integrability_moments_of_interest}
Let $\mu \in \mathcal{Y}^r([0,T]\times\Omega;\bm{Y}\times\bm{Z})$. Suppose that $\mu$ satisfies the integration-by-parts formula. Then, one has
\begin{equation}
\label{eq:bounded_m1_and_m2}
    m_1,m_2 \in W^{1,r}((0,T)\times\Omega) \cap L^\infty((0,T)\times\Omega).
\end{equation}
In particular, $m_1$ and $m_2$ admit two continuous representatives (still denoted $m_1$ and $m_2$) in the following sense:
\begin{equation}
\label{eq:continuous_m1_and_m2}
    m_1,m_2 \in C([0,T];L^r(\Omega)) \cap L^2(0,T;C(\overline{\Omega})).
\end{equation}
Furthermore, the integral source terms satisfy
\begin{equation}
    m_f\in L^\infty((0,T)\times\Omega)
\end{equation}
and
\begin{equation}
\hat{m}_f\in L^{r/2}((0,T)\times\Omega).
\end{equation}
\end{proposition}
\begin{proof}
The fact that $m_1$ and $m_2$ satisfy \eqref{eq:bounded_m1_and_m2} and \eqref{eq:continuous_m1_and_m2} comes from Proposition \ref{prop:existence_of_weak_derivative_for_Young_measures} applied to $\beta(y)=y$ and $\beta(y)=y^2$, respectively.

Now, let us observe that, since $\bm{Y}$ is compact and since $f$ satisfies \eqref{eq:boundedness_f}, one has for almost all $(t,x)\in [0,T]\times\Omega$,
\begin{equation*}
    \left\lvert m_f(t,x) \right\rvert = \left\lvert \inner{\mutx}{f(t,x,y)} \right\rvert \leq \inner{\mutx}{\left\lvert f(t,x,y)\right\rvert} \leq M_{\bm{Y}} \inner{\mutx}{1},
\end{equation*}
for some constant $M_{\bm{Y}} \geq 0$. Since $\mutx$ is a probability measure, it follows that $\inner{\mutx}{1} = 1$. Therefore, $\left\lvert m_f(t,x) \right\rvert \leq M_{\bm{Y}}$ for almost all $(t,x)\in[0,T]\times\Omega$. Thus, $m_f \in L^\infty((0,T)\times\Omega)$.

Finally, in virtue of Proposition \ref{prop:integrability_Young_measures}, we have $2\inner{\mutx}{\left\lVert \overline{z} \right\rVert^2} \in L^{r/2}((0,T)\times\Omega)$. Because $2\inner{\mutx}{yf(t,x,y)} \in L^\infty((0,T)\times\Omega)$ for the same reason as $m_f$, it follows that $\hat{m}_f\in L^{r/2}((0,T)\times\Omega)$.
\end{proof}
We now introduce the notion of energy measure-valued solutions to the PDE \eqref{eq:semilinear_PDE}. Such solutions are Young measures satisfying a variational identity analogous to the weak formulation satisfied by $y^*$ in \eqref{eq:weak_solution}, together with energy identities corresponding to \eqref{eq:weak_formulation_energy} and \eqref{eq:weak_formulation_dissipation}.
\begin{definition}[Energy measure-valued solutions]
\label{def:emv_solutions} 
Let $\mu \in \mathcal{Y}^r([0,T]\times\Omega;\bm{Y}\times\bm{Z})$. We say that $\mu$ is an \emph{energy measure-valued (emv) solution} to the PDE \eqref{eq:semilinear_PDE} if it satisfies the integration-by-parts formula and the following moment identities: for every $v\in H_0^1(\Omega)$ and every $\varphi\in C_c^\infty((0,T))$,
\begin{subequations}
\label{eq:mv_solutions_constraints}
\begin{gather}
\int_0^T \int_\Omega \varphi(t) \Bigl[v(x) \partial_tm_1(t,x)+\nabla_xv(x)\cdot\nabla_xm_1(t,x) -v(x) m_f(t,x)\Bigr] \ud x \ud t = 0,\label{eq:weak_formulation_m1}\\
m_1(t,\cdot) = 0, \text{ on $\partial\Omega$ for a.e. } t\in[0,T], \label{eq:boundary_condition_m1}\\
m_1(0,\cdot)= y_0, \text{ in $L^r(\Omega)$}, \label{eq:initial_condition_m1}
\end{gather}
\end{subequations}
and
\begin{subequations}
\label{eq:emv_solutions_constraints}
\begin{gather}
\int_0^T \int_\Omega \varphi(t)\Bigl[ v(x) \partial_tm_2(t,x)+\nabla_xv(x)\cdot\nabla_xm_2(t,x) -v(x) \hat{m}_f(t,x)\Bigr] \ud x \ud t = 0,\label{eq:weak_formulation_m2}\\
m_2(t,\cdot) = 0, \text{ on $\partial\Omega$ for a.e. } t\in[0,T], \label{eq:boundary_condition_m2}\\
m_2(0,\cdot)= y_0^2, \text{ in $L^r(\Omega)$}, \label{eq:initial_condition_m2}
\end{gather}
\end{subequations}
and
\begin{equation}
\label{eq:moment_dissipation_z0}
\int_0^T\int_\Omega \varphi(t)\inner{\mutx}{z_0^2} - \varphi'(t)\inner{\mutx}{\frac{1}{2}\left\lVert \overline{z} \right\rVert^2}- \varphi(t)\inner{\mutx}{z_0f(t,x,y)} \ud x \ud t = 0.
\end{equation}
\end{definition}
\begin{remark}
Notice that \eqref{eq:weak_formulation_m1} is equivalent to requiring that for any $v\in H^1_0(\Omega)$ and for almost all $t\in[0,T]$,
\begin{equation}
\label{eq:weak_formulation_m1_ae_t}
\int_\Omega v(x) \partial_tm_1(t,x)+\nabla_xv(x)\cdot\nabla_xm_1(t,x) -v(x) m_f(t,x)\ud x=0.
\end{equation}
Furthermore, since $\mu$ has finite $r$-moments in Definition \ref{def:emv_solutions}, we know from Proposition \ref{prop:integrability_moments_of_interest} that $m_1\in W^{1,r}((0,T)\times\Omega)\cap L^\infty((0,T)\times\Omega)$ and $m_f\in L^\infty((0,T)\times\Omega)$. Thus the space integral in \eqref{eq:weak_formulation_m1} and \eqref{eq:weak_formulation_m1_ae_t} is well defined for all $v\in H^1_0(\Omega)$. Finally, the prescriptions of the boundary condition \eqref{eq:boundary_condition_m1} and the initial condition \eqref{eq:initial_condition_m1} make sense in virtue of the continuity result \eqref{eq:continuous_m1_and_m2} satisfied by $m_1$.

A Young measure satisfying only the constraints
\eqref{eq:mv_solutions_constraints} involving $m_1$ is usually referred to as a \emph{measure-valued (mv) solution} to the original PDE. However, as discussed below in Remark \ref{remark:mv_solutions_are_not_enough}, this notion is not
restrictive enough to imply concentration.
\end{remark}
\begin{remark}
Notice that \eqref{eq:weak_formulation_m2} is equivalent to requiring that for any $v\in H^1_0(\Omega)$ and for almost all $t\in[0,T]$,
\begin{equation}
\label{eq:weak_formulation_m2_ae_t}
\int_\Omega  v(x) \partial_tm_2(t,x)+\nabla_xv(x)\cdot\nabla_xm_2(t,x) -v(x) \hat{m}_f(t,x) \ud x=0.
\end{equation}
Furthermore, since $\mu$ has finite $r$-moments in Definition \ref{def:emv_solutions}, we know from Proposition \ref{prop:integrability_moments_of_interest} that $m_2\in W^{1,r}((0,T)\times\Omega) \subset H^1((0,T)\times\Omega)$. Therefore, one has $v\partial_tm_2(t,\cdot)\in L^1(\Omega)$ and $ \nabla_x v \cdot \nabla_x m_2(t,\cdot) \in L^1(\Omega)$ for any $v \in H^1_0(\Omega)$ and for almost all $t\in[0,T]$. By Proposition \ref{prop:integrability_moments_of_interest} again, one has $\hat{m}_f\in L^{r/2}((0,T)\times\Omega)$. Therefore, $\hat{m}_f(t,\cdot) \in L^{r/2}(\Omega)$ for almost all $t\in[0,T]$, and it follows from Proposition \ref{prop:properties_sobolev_Lr2} applied to $\alpha = 1/2$ that $\hat{m}_f(t,\cdot) v \in L^1(\Omega)$ for any $v \in H^1_0(\Omega)$. Thus, the space integral in \eqref{eq:weak_formulation_m2} and \eqref{eq:weak_formulation_m2_ae_t} is well defined for any $v\in H^1_0(\Omega)$ since all the terms are integrable.

The prescriptions of the boundary condition \eqref{eq:boundary_condition_m2} and the initial condition \eqref{eq:initial_condition_m2} make sense in virtue of the continuity result \eqref{eq:continuous_m1_and_m2} satisfied by $m_2$.

Now, we proceed to justify that \eqref{eq:moment_dissipation_z0} is well defined. First, let us observe that since $\bm{Y}$ is compact and since $f$ satisfies \eqref{eq:boundedness_f}, one has for almost all $(t,x)\in [0,T]\times\Omega$,
\begin{equation*}
    \left\lvert \inner{\mutx}{z_0f(t,x,y)} \right\rvert \leq \inner{\mutx}{\left\lvert z_0 f(t,x,y)\right\rvert} \leq M_{\bm{Y}} \inner{\mutx}{\left\lvert z_0 \right\rvert},
\end{equation*}
for some constant $M_{\bm{Y}} \geq 0$. Using the fact that $\mu$ has finite $r$-moments, one immediately obtains that $\inner{\mutx}{\left\lvert z_0 \right\rvert} \in L^r((0,T)\times\Omega)$. It follows from the estimate above that $\inner{\mutx}{ z_0 f(t,x,y) } \in L^r((0,T)\times\Omega) \subset L^1((0,T)\times \Omega)$ is integrable. Finally, one can use Proposition \ref{prop:integrability_Young_measures} to obtain that $\inner{\mutx}{z_0^2}, \inner{\mutx}{\left\lVert\overline{z}\right\rVert^2}\in L^{r/2}((0,T)\times\Omega)$. Since $r\geq2$, it follows that $\inner{\mutx}{z_0^2}, \inner{\mutx}{\left\lVert\overline{z}\right\rVert^2}\in L^{1}((0,T)\times\Omega)$. Therefore, \eqref{eq:moment_dissipation_z0} is well defined since all the terms are integrable.
\end{remark}
\begin{remark}
\label{remark:mv_solutions_are_not_enough}
The constraints \eqref{eq:mv_solutions_constraints} appearing in Definition~\ref{def:emv_solutions} are not sufficient, in general, to imply concentration. To illustrate this point, consider the particular case where
\begin{equation*}
f(t,x,y)=y^3 \qquad\text{and}\qquad u_0=0
\end{equation*}
in the PDE \eqref{eq:semilinear_PDE}. In this case, the corresponding unique solution is $y^*=0$.

Let $g\in C_c^\infty((0,T)\times\Omega)$ be nonidentically zero, and define $\mu$ as the parametrized family of probability measures
\begin{equation*}
    \mutx \coloneq \frac{1}{2} \left( \delta_{(g(t,x),\partial_t g(t,x),\nabla_x g(t,x))} + \delta_{(-g(t,x),-\partial_t g(t,x),-\nabla_x g(t,x))} \right).
\end{equation*}
Since $g$ is smooth and compactly supported, the measure $\mu$ has finite $r$-moments and satisfies the null boundary condition and initial condition \eqref{eq:boundary_condition_m1} and \eqref{eq:initial_condition_m1}. Moreover, since $g$ satisfies the integration-by-parts formula for functions, it follows that $\mu$ satisfies the integration-by-parts formula for Young measures. Finally, the weak formulation on Young measures \eqref{eq:weak_formulation_m1} is also satisfied. Indeed, by symmetry, one has
\begin{equation*}
    m_1(t,x) = \frac{1}{2}g(t,x) +\frac{1}{2}\left(-g(t,x) \right) = 0,
\end{equation*}
and since $f(y)=y^3$ is odd, one has
\begin{equation*}
    \inner{\mutx}{f(t,x,y)} = \frac{1}{2} g(t,x)^3 + \frac{1}{2} \left(-g(t,x)\right)^3 = 0.
\end{equation*}
Thus all the constraints \eqref{eq:mv_solutions_constraints} in Definition~\ref{def:emv_solutions} are fulfilled.

However, $\mutx$ is not concentrated on the solution $y^*=0$ on the set where $g(t,x)\neq 0$, which has positive measure for a suitable choice of $g$. This motivates the introduction of the second-order moment conditions \eqref{eq:emv_solutions_constraints} and \eqref{eq:moment_dissipation_z0}, which are designed to rule out such non-concentrated symmetric configurations.
\end{remark}

We are now ready to state the main result of this work.
\begin{theorem}[Concentration of emv solutions]
\label{theorem:concentration_emv_solutions}
Suppose that Assumption~\ref{assumption:existence_weak_solution} holds, and let $y^*$ be the corresponding weak solution of the PDE. Then every emv solution $\mu$ to the PDE \eqref{eq:semilinear_PDE} is concentrated on $y^*$ and its first-order derivatives, namely
\begin{equation*}
    \mutx = \delta_{\left(y^*(t,x),\partial_t y^*(t,x),\nabla_x y^*(t,x)\right)}  \quad \text{for a.e. } (t,x)\in [0,T]\times\Omega .
\end{equation*}
\end{theorem}

\section{Proof of Theorem \ref{theorem:concentration_emv_solutions}}
\label{sec:proof}

In order to prove Theorem \ref{theorem:concentration_emv_solutions}, we suppose that Assumption \ref{assumption:existence_weak_solution} holds. By Proposition~\ref{prop:uniqueness_weak_solution}, the corresponding weak solution $y^*$ is unique. Let $\mu \in \mathcal{Y}^r([0,T]\times\Omega;\bm{Y}\times\bm{Z})$ be an emv solution to the PDE \eqref{eq:semilinear_PDE}.

We begin by introducing a quantity that measures the deviation of the Young measure $\mu$ from the Dirac mass concentrated at $y^*(t,x)$. This squared-error density will be the main object in the concentration argument.

\begin{definition}[Squared-error density]
The \emph{squared-error density $w$} is defined for almost all $(t,x) \in [0,T]\times\Omega$ by
\begin{equation}
\label{eq:definition_w}
    w(t,x) \coloneq \inner{\mutx}{\left(y - y^*(t,x) \right)^2} = m_2(t,x) -2y^*(t,x)m_1(t,x) + (y^*(t,x))^2.
\end{equation} 
\end{definition}
The first step of the proof consists in proving the following lemma.
\begin{lemma}
The squared-error density $w$ satisfies
\begin{equation}
\label{eq:space_w}
    w \in H^1(0,T;L^r(\Omega)) \cap L^2(0,T;W_0^{1,r}(\Omega)).
\end{equation}
Moreover, for any $v\in H^1_0(\Omega)$ and for almost all $t\in[0,T]$, it holds that
\begin{multline}
\label{eq:PDE_w}
    \int_\Omega \Bigl[v(x) \partial_t w(t,x) + \nabla_x v(x) \cdot \nabla_x w(t,x) +2v(x)\inner{\mutx}{\left\lVert \overline{z}-\nabla_xy^*(t,x) \right\rVert^2}\\
    -2v(x) \inner{\mutx}{\left(y-y^*(t,x)\right)\left(f(t,x,y)-f(t,x,y^*(t,x))\right)} \Bigr]\ud x = 0.
\end{multline}
Finally, $w$ satisfies the null initial condition in $L^r(\Omega)$
\begin{align}
\label{eq:initial_condition_w}
    w(0,\cdot)=0.
\end{align}
\end{lemma}
\begin{proof}
Let us prove \eqref{eq:space_w} first. We know from Proposition \ref{prop:integrability_moments_of_interest} that
\begin{equation*}
m_1,m_2 \in W^{1,r}((0,T)\times\Omega) \subset H^1(0,T;L^r(\Omega))\cap L^2(0,T;W^{1,r}(\Omega)),
\end{equation*}
where the embedding of the two function spaces above was previously shown in the proof of Proposition \ref{prop:existence_of_weak_derivative_for_Young_measures}. Since $m_1$ and $m_2$ satisfy respectively the null boundary conditions \eqref{eq:boundary_condition_m1} and \eqref{eq:boundary_condition_m2}, it follows that
\begin{equation*}
m_1,m_2 \in H^1(0,T;L^r(\Omega))\cap L^2(0,T;W_0^{1,r}(\Omega)).
\end{equation*}

\noindent Next, using that $y^* \in X_r = H^1(0,T;W^{1,r}_0(\Omega))$, we obtain
\begin{equation*}
y^* m_1 \in H^1(0,T;L^r(\Omega))\cap L^2(0,T;W^{1,r}_0(\Omega)).
\end{equation*}
The proof is very similar to that of Proposition \ref{prop:product_in_Xr}; it relies on the fact that $W^{1,r}_0(\Omega)$ is a Banach algebra and that $$\partial_t(m_1 y^*) = m_1 \partial_t y^* + y^* \partial_t m_1 \in L^2(0,T;L^r(\Omega)),$$ since both $y^*$ and $m_1$ are in $L^\infty((0,T)\times\Omega)$.

\noindent Finally, we know from Proposition \ref{prop:product_in_Xr} that $(y^*)^2 \in X_r$. Since the embedding $$X_r = H^1(0,T;W^{1,r}_0(\Omega)) \subset H^1(0,T;L^r(\Omega))\cap L^2(0,T;W^{1,r}_0(\Omega))$$ holds, it follows that all the terms in \eqref{eq:definition_w} are in $H^1(0,T;L^r(\Omega))\cap L^2(0,T;W^{1,r}_0(\Omega))$, which proves \eqref{eq:space_w}.

Now, we proceed to prove \eqref{eq:PDE_w}. Let $v\in H^1_0(\Omega)$ and consider $\tilde{v}_t\coloneq v y^*(t,\cdot)$ and $\hat{v}_t\coloneq v m_1(t,\cdot)$ for almost all $t\in[0,T]$. Since $y^*\in X_r$ and $m_1 \in W^{1,r}((0,T)\times\Omega)$, we have $y^*(t,\cdot) \in W^{1,r}(\Omega)$ and $m_1(t,\cdot) \in W^{1,r}(\Omega)$ for almost all $t\in[0,T]$. By Proposition \ref{prop:properties_sobolev_W1r}, it follows that $\tilde{v}_t \in H^1_0(\Omega)$ and $\hat{v}_t \in H^1_0(\Omega)$ for almost all $t\in[0,T]$.

\noindent Using $\tilde{v}_t$ as a test function in equation \eqref{eq:weak_formulation_m1_ae_t} and the product rule in space yields for almost all $t\in[0,T]$,
\begin{multline*}
    \int_\Omega \Bigl[ v(x)y^*(t,x) \partial_tm_1(t,x)+\left(y^*(t,x)\nabla_xv(x) + v(x)\nabla_xy^*(t,x) \right) \cdot\nabla_xm_1(t,x)\\
    -v(x)y^*(t,x) m_f(t,x)\Bigr] \ud x = 0.
\end{multline*}
Doing the same with $\hat{v}_t$ in \eqref{eq:weak_solution_ae_t} gives for almost all $t\in[0,T]$,
\begin{multline*}
\int_\Omega \Bigl[ v(x)m_1(t,x) \partial_ty^*(t,x) + \left(m_1(t,x)\nabla_xv(x) + v(x)\nabla_xm_1(t,x) \right) \cdot \nabla_x y^*(t,x)\\
- v(x)m_1(t,x) f(t,x,y^*(t,x)) \Bigr] \ud x = 0.
\end{multline*}
Adding the previous two equations and applying the product rule in time and space to $y^*m_1$ yields for almost all $t\in[0,T]$,
\begin{multline*}
    \int_\Omega \Bigl[ v(x) \partial_t(y^*m_1)(t,x) + \nabla_x v(x) \cdot \nabla_x (y^*m_1)(t,x) + 2v(x) \nabla_xy^*(t,x) \cdot \nabla_x m_1(t,x)\\
    -v(x)y^*(t,x) m_f(t,x)- v(x)m_1(t,x) f(t,x,y^*(t,x)) \Bigr] \ud x = 0.
\end{multline*}
Then, we can use the expressions $m_1(t,x)=\inner{\mutx}{y}$ and $m_f(t,x)=\inner{\mutx}{f(t,x,y)}$ to obtain for almost all $t\in[0,T]$,
\begin{multline}
\label{eq:equation_y_and_m1}
    \int_\Omega \Bigl[ v(x) \partial_t(y^*m_1)(t,x) + \nabla_x v(x) \cdot \nabla_x (y^*m_1)(t,x) + 2v(x) \nabla_xy^*(t,x) \cdot \nabla_x m_1(t,x)\\
    -v(x)\inner{\mutx}{y^*(t,x)f(t,x,y) + yf(t,x,y^*(t,x))} \Bigr] \ud x = 0.
\end{multline}
Since $y^*\in X_r$ is a weak solution to the PDE \eqref{eq:semilinear_PDE}, it satisfies the energy identity \eqref{eq:weak_formulation_energy_ae_t}, which we recall here for clarity. For almost all $t\in[0,T]$, it holds that
\begin{multline}
\label{eq:weak_formulation_energy_again}
    \int_\Omega \Bigl[ v(x) \partial_t((y^*)^2)(t,x) +\nabla_xv\cdot\nabla_x((y^*)^2)(t,x) - 2v(x)\inner{\mutx}{y^*(t,x)f(t,x,y^*(t,x))}\\
    +2v(x) \left\lVert \nabla_x y^*(t,x) \right\rVert^2 \Bigr] \ud x = 0.
\end{multline}
Notice that because $\mutx$ is a probability measure, we could replace $y^*(t,x)f(t,x,y^*(t,x))$ by $\inner{\mutx}{y^*(t,x)f(t,x,y^*(t,x))}$ in \eqref{eq:weak_formulation_energy_ae_t}.

\noindent Next, since $\mu$ is an emv solution to the PDE, $m_2$ satisfies the energy identity \eqref{eq:weak_formulation_m2_ae_t}. It writes for almost all $t\in[0,T]$,
\begin{multline}
\label{eq:weak_formulation_m2_again}
\int_\Omega \Bigl[ v(x) \partial_tm_2(t,x)+\nabla_xv(x)\cdot\nabla_xm_2(t,x) -2v(x)\inner{\mutx}{yf(t,x,y)}\\
+2v(x)\inner{\mutx}{\left\lVert \overline{z} \right\rVert^2} \Bigr] \ud x = 0.
\end{multline}
Notice that we have replaced $\hat{m}_f(t,x)$ by its definition $2\inner{\mutx}{yf(t,x,y)} - 2\inner{\mutx}{\left\lVert \overline{z} \right\rVert^2}$ in \eqref{eq:weak_formulation_m2_ae_t}. Taking $\eqref{eq:weak_formulation_m2_again} -2\times\eqref{eq:equation_y_and_m1} + \eqref{eq:weak_formulation_energy_again}$ gives for almost all $t\in[0,T]$,
\begin{align*}
0 &= \int_\Omega \Bigl[
      v(x)\partial_t\left(m_2 - 2y^*m_1 + (y^*)^2\right)(t,x)
      + \nabla_x v(x) \cdot \nabla_x\left(m_2 - 2y^*m_1 + (y^*)^2\right)(t,x)\\
 &\quad- 2v(x)\inner{\mutx}{yf(t,x,y)-y^*(t,x)f(t,x,y)-yf(t,x,y^*(t,x))+y^*(t,x)f(t,x,y^*(t,x))}\\
&\quad+ 2v(x)\left(\inner{\mutx}{\left\lVert \overline{z} \right\rVert^2}
      - 2\nabla_x y^*(t,x) \cdot \nabla_x m_1(t,x) + \left\lVert \nabla_x y^*(t,x) \right\rVert^2 \right)
   \Bigr]\ud x.
\end{align*}
Using the definition of $w$ \eqref{eq:definition_w} in the equation above yields for almost all $t\in[0,T]$,
\begin{align}
\label{eq:intermediate_pde_w}
0 &= \int_\Omega \Bigl[
      v(x)\partial_tw(t,x)
      + \nabla_x v(x) \cdot \nabla_xw(t,x) \notag \\
&\quad- 2v(x)\inner{\mutx}{\left(y-y^*(t,x)\right)\left(f(t,x,y)-f(t,x,y^*(t,x))\right)} \notag\\
&\quad+ 2v(x)\left(\inner{\mutx}{\left\lVert \overline{z} \right\rVert^2}
      - 2\nabla_x y^*(t,x) \cdot \nabla_x m_1(t,x) + \left\lVert \nabla_x y^*(t,x) \right\rVert^2 \right)
   \Bigr]\ud x.
\end{align}
Now, let us observe that for almost all $(t,x) \in [0,T]\times\Omega$,
\begin{align}
\label{eq:inner_intermediate_squared_norm_identity}
    \inner{\mutx}{\left\lVert \overline{z}-\nabla_xy^*(t,x) \right\rVert^2} &= \inner{\mutx}{\left\lVert \overline{z} \right\rVert^2 -2 \nabla_xy^*(t,x) \cdot \overline{z} + \left\lVert \nabla_x y^*(t,x) \right\rVert^2} \notag \\
    &= \inner{\mutx}{\left\lVert \overline{z} \right\rVert^2} -2\nabla_xy^*(t,x) \cdot \inner{\mutx}{\overline{z}} + \left\lVert \nabla_x y^*(t,x) \right\rVert^2,
\end{align}
where we have used the fact that $\mutx$ is a probability measure. Since $\mu \in \mathcal{Y}^r([0,T]\times\Omega;\bm{Y}\times\bm{Z})$ satisfies the integration-by-parts formula, we know from Proposition \ref{prop:existence_of_weak_derivative_for_Young_measures} that $\nabla_xm_1$ exists in $L^r((0,T)\times\Omega)$ and is given by $\nabla_xm_1 = \inner{\mutx}{\overline{z}}$. Thus, \eqref{eq:inner_intermediate_squared_norm_identity} becomes for almost all $(t,x) \in [0,T]\times\Omega$
\begin{equation}
\label{eq:inner_final_squared_norm_identity}
\inner{\mutx}{\left\lVert \overline{z}-\nabla_xy^*(t,x) \right\rVert^2} = \inner{\mutx}{\left\lVert \overline{z} \right\rVert^2}- 2\nabla_x y^*(t,x) \cdot \nabla_x m_1(t,x) + \left\lVert \nabla_x y^*(t,x) \right\rVert^2.
\end{equation}
Replacing \eqref{eq:inner_final_squared_norm_identity} in \eqref{eq:intermediate_pde_w} thus gives for any $v \in H^1_0(\Omega)$ and for almost all $t \in [0,T]$,
\begin{multline*}
0 = \int_\Omega \Bigl[
      v(x)\partial_tw(t,x)
      + \nabla_x v(x) \cdot \nabla_xw(t,x) + 2v(x)\inner{\mutx}{\left\lVert \overline{z}-\nabla_xy^*(t,x) \right\rVert^2}\\
      - 2v(x)\inner{\mutx}{\left(y-y^*(t,x)\right)\left(f(t,x,y)-f(t,x,y^*(t,x))\right)}
   \Bigr]\ud x,
\end{multline*}
which is exactly \eqref{eq:PDE_w}.

\noindent Finally, the null initial condition \eqref{eq:initial_condition_w} follows directly from the definition of $w$ \eqref{eq:definition_w} and the fact that $y^*(0,\cdot) = m_1(0,\cdot) = y_0$ and $m_2(0,\cdot) = y_0^2$.
\end{proof}
From the formulation \eqref{eq:PDE_w}, we can deduce a variational inequality satisfied by $w$.

\noindent Indeed, it follows from \eqref{eq:compact_Y} that $y^*(t,x)\in\bm{Y}$ for all $(t,x) \in [0,T]\times\overline{\Omega}$. Since $\bm{Y}$ is compact, the locally one-sided Lipschitz assumption \eqref{eq:OSL_f} on $f$ implies that there exists $L_{\bm{Y}}\geq0$ such that for any $y \in \bm{Y}$ and almost all $(t,x) \in [0,T]\times\Omega$,
\begin{align*}
\left(y-y^*(t,x)\right)\left(f(t,x,y)-f(t,x,y^*(t,x))\right) \leq L_{\bm{Y}} \left(y-y^*(t,x)\right)^2.
\end{align*}
Thus, we have for almost all $(t,x) \in [0,T]\times\Omega$,
\begin{align*}
\inner{\mutx}{\left(y-y^*(t,x)\right)\left(f(t,x,y)-f(t,x,y^*(t,x))\right)} &\leq L_{\bm{Y}} \inner{\mutx}{\left(y-y^*(t,x)\right)^2}\\
&= L_{\bm{Y}}w(t,x).
\end{align*}
Let $v^+ \in H^1_0(\Omega;\R_+)$ be a nonnegative test function. Multiplying the equation above by $v^+$ and integrating in space gives for almost all $t\in[0,T]$,
\begin{equation}
\label{eq:proof_estimate_osl_intermediate}
    \int_\Omega v^+(x) \inner{\mutx}{\left(y-y^*(t,x)\right)\left(f(t,x,y)-f(t,x,y^*(t,x))\right)} \ud x \leq\!\! \int_\Omega \! L_{\bm{Y}} v^+(x)w(t,x) \ud x.
\end{equation}
Moreover, observe that for almost all $t\in[0,T]$,
\begin{equation}
\label{eq:proof_estimate_norms_intermediate}
    0 \leq \int_\Omega v^+(x) \inner{\mutx}{\left\lVert \overline{z} - \nabla_xy^*(t,x) \right\rVert^2} \ud x.
\end{equation}
Using the estimates \eqref{eq:proof_estimate_osl_intermediate} and \eqref{eq:proof_estimate_norms_intermediate} in \eqref{eq:PDE_w} finally yields the variational inequality satisfied by $w$. It writes for any $v^+ \in H^1_0(\Omega;\R_+)$ and for almost all $t \in [0,T]$ as follows:
\begin{equation}
\label{eq:variational_inequality_w}
\int_\Omega v^+(x) \partial_t w(t,x) + \nabla_x v^+(x) \cdot \nabla_x w(t,x) -2L_{\bm{Y}}v^+(x)w(t,x) \ud x \leq 0.
\end{equation}
We are now ready to prove the following lemma.
\begin{lemma} The squared-error density is null almost everywhere on $[0,T]\times\Omega$, i.e.,
\begin{equation}
\label{eq:w_equal_0}
    w(t,x)=0, \quad \text{a.e. } (t,x)\in[0,T]\times\Omega.
\end{equation}
\end{lemma}
\begin{proof}
Let us consider for almost all $(t,x) \in [0,T]\times\Omega$,
\begin{equation}
\label{eq:definition_w_hat}
    \hat{w}(t,x) = e^{-2tL_{\bm{Y}}} w(t,x).
\end{equation}
Since $\left\lvert \hat{w}(t,x) \right\rvert \leq \left\lvert w(t,x) \right\rvert$, it follows from \eqref{eq:space_w} that
\begin{equation}
\label{eq:space_w_hat}
    \hat{w} \in H^1(0,T;L^r(\Omega)) \cap L^2(0,T;W^{1,r}_0(\Omega)).
\end{equation}
Furthermore, one has for almost all $(t,x) \in [0,T]\times\Omega$,
\begin{gather*}
\partial_t \hat{w}(t,x) = e^{-2tL_{\bm{Y}}} \left( \partial_t w(t,x) -2L_{\bm{Y}} w(t,x) \right),\\
\nabla_x \hat{w}(t,x) = e^{-2tL_{\bm{Y}}} \nabla_x w(t,x).
\end{gather*}
Therefore, the variational inequality \eqref{eq:variational_inequality_w} on $w$ becomes for any $v^+\in H^1_0(\Omega;\R_+)$ and for almost all $t\in[0,T]$,
\begin{equation}
\label{eq:variational_inequality_w_hat}
\int_\Omega v^+(x) \partial_t \hat{w}(t,x) + \nabla_x v^+(x) \cdot \nabla_x \hat{w}(t,x) \ud x \leq 0.
\end{equation}
Let $g\in C^\infty_c((0,T)\times\Omega; \R_+)$ be nonnegative and consider the following heat equation
\begin{align}
\label{eq:heat_equation_for_test_function}
\begin{dcases}
    \partial_t \varphi_g(t,x) - \Delta_x\varphi_g(t,x) = g(T-t,x), &\quad t \in (0,T), \, x \in \Omega,\\
    \varphi_g(t,x) = 0, &\quad t\in(0,T), \, x \in \partial \Omega,\\
    \varphi_g(0,x) = 0, &\quad x\in\Omega.
\end{dcases}
\end{align}
We denote by $\varphi_{g} \in H^1(0,T;H^{-1}(\Omega)) \cap L^2(0,T;H^1_0(\Omega)) \subset C([0,T];L^2(\Omega))$ the unique weak solution to the PDE  \eqref{eq:heat_equation_for_test_function} (existence, uniqueness, regularity and time continuity are stated in \cite[Theorem 4.30]{weak_solutions_and_weak_max_principle_heat_equation}). It satisfies for any $v\in H^1_0(\Omega)$ and almost all $t\in [0,T]$,
\begin{equation}
\label{eq:weak_formulation_phi_g}
    \int_\Omega v(x)g(T-t,x) \ud x = \bigl(v,\partial_t \varphi_g(t,\cdot)\bigr)_{H^1_0(\Omega),H^{-1}(\Omega)} + \int_\Omega \nabla_x v(x) \cdot \nabla_x \varphi_g(t,x) \ud x,
\end{equation}
where $\bigl( \cdot , \cdot \bigr)_{H^1_0(\Omega),H^{-1}(\Omega)}$ denotes the duality pairing between $H_0^1(\Omega)$ and $H^{-1}(\Omega)$. Since $\hat{w}$ satisfies \eqref{eq:space_w_hat}, it follows that $v_t = \hat{w}(T-t,\cdot) \in W^{1,r}_0(\Omega) \subset H^1_0(\Omega)$ for almost all $t\in[0,T]$. Therefore, using $v_t$ as a test function in \eqref{eq:weak_formulation_phi_g} gives for almost all $t\in[0,T]$,
\begin{multline*}
    \int_\Omega \hat{w}(T-t,x)g(T-t,x) \ud x = \bigl(\hat{w}(T-t,\cdot),\partial_t \varphi_g(t,\cdot)\bigr)_{H^1_0(\Omega),H^{-1}(\Omega)}\\
    + \int_\Omega \nabla_x \hat{w}(T-t,x) \cdot \nabla_x \varphi_g(t,x) \ud x.
\end{multline*}
Integrating the equation above with respect to time and making the change of variables $t'=T-t$ on the left-hand side yields
\begin{multline}
\label{eq:weak_formulation_phi_g_integrated_in_time}
    \int_0^T \int_\Omega \hat{w}(t,x)g(t,x) \ud x \ud t = \int_0^T \bigl(\hat{w}(T-t,\cdot),\partial_t \varphi_g(t,\cdot)\bigr)_{H^1_0(\Omega),H^{-1}(\Omega)} \ud t\\
    + \int_0^T \int_\Omega \nabla_x \hat{w}(T-t,x) \cdot \nabla_x \varphi_g(t,x) \ud x \ud t.
\end{multline}
Next, since both $\hat{w}$ and $\varphi_g$ belong to $H^1(0,T;H^{-1}(\Omega)) \cap L^2(0,T;H^1_0(\Omega))$, we may apply the corresponding integration-by-parts formula (see \cite[Lemma 5.1]{IBP_Lions_lemma} for instance) to the pair $(\hat{w},\varphi_g)$, obtaining
\vspace{-0.02cm}
\begin{multline*}
\int_0^T \bigl(\hat{w}(T-t,\cdot),\partial_t \varphi_g(t,\cdot)\bigr)_{H^1_0(\Omega),H^{-1}(\Omega)} \ud t = -\int_0^T \bigl(\varphi_g(t,\cdot), -\partial_t\hat{w}(T-t,\cdot)\bigr)_{H^1_0(\Omega),H^{-1}(\Omega)} \ud t\\
+ \int_\Omega \hat{w}(0,x)\varphi_g(T,x)-\hat{w}(T,x)\varphi_g(0,x) \ud x.
\end{multline*}
Recall from \eqref{eq:initial_condition_w} and \eqref{eq:heat_equation_for_test_function} that $\hat{w}(0,\cdot) = 0$ and $\varphi_g(0,\cdot)=0$. Therefore, the equation above simply writes
\begin{equation}
\label{eq:intermediate_IBP_H-1}
\int_0^T\! \bigl(\hat{w}(T-t,\cdot),\partial_t \varphi_g(t,\cdot)\bigr)_{H^1_0(\Omega),H^{-1}(\Omega)} \ud t = \int_0^T \!\bigl(\varphi_g(t,\cdot), \partial_t\hat{w}(T-t,\cdot)\bigr)_{H^1_0(\Omega),H^{-1}(\Omega)} \ud t.
\end{equation}
Finally, it follows from \eqref{eq:space_w_hat} that $\partial_t \hat{w}(T-t,\cdot) \in L^r(\Omega) \subset L^2(\Omega)$ for almost all $t\in[0,T]$. Thus, the dual pairing in the right-hand side of \eqref{eq:intermediate_IBP_H-1} is simply the inner product between two functions in $L^2(\Omega)$, which gives
\begin{equation}
\label{eq:final_IBP_H-1}
\int_0^T \bigl(\hat{w}(T-t,\cdot),\partial_t \varphi_g(t,\cdot)\bigr)_{H^1_0(\Omega),H^{-1}(\Omega)} \ud t = \int_0^T \int_\Omega \varphi_g(t,x) \partial_t \hat{w}(T-t,x) \ud x  \ud t.
\end{equation}
Thus, one may use the identity \eqref{eq:final_IBP_H-1} in the weak formulation \eqref{eq:weak_formulation_phi_g_integrated_in_time} to obtain
\begin{align}
\label{eq:identity_phi_g_and_w_hat}
    \int_0^T \!\!\!\!\int_\Omega \hat{w}(t,x)g(t,x) \ud x \ud t &=\! \int_0^T \!\!\!\!\int_\Omega \!\varphi_g(t,x) \partial_t \hat{w}(T\!-\!t,x) \!+\! \nabla_x \hat{w}(T\!-\!t,x) \!\cdot\! \nabla_x \varphi_g(t,x) \ud x \ud t \notag\\
    &=\! \int_0^T \!\!\!\!\int_\Omega \!\varphi_g(T\!-\!t,x) \partial_t \hat{w}(t,x) \!+\! \nabla_x \hat{w}(t,x) \!\cdot\! \nabla_x \varphi_g(T\!-\!t,x) \ud x \ud t.
\end{align}
Let us note that, since $\varphi_g$ is a weak solution to the heat equation \eqref{eq:heat_equation_for_test_function} with a nonnegative source term $g$, the weak maximum principle (we refer to \cite[Proposition 4.34]{weak_solutions_and_weak_max_principle_heat_equation} for the homogeneous case $g=0$; the proof applies without modification when $g\ge 0$) implies that
\begin{equation*}
\varphi_g(T-t,\cdot)\ge 0, \quad \text{a.e. in }\Omega,\quad \forall t\in[0,T].
\end{equation*}
Therefore, one has $v_t^+ \coloneq \varphi_g(T-t,\cdot) \in H^1_0(\Omega,\R_+)$ for almost all $t\in[0,T]$. Then, using $v_t^+$ as a test function in \eqref{eq:variational_inequality_w_hat} yields for almost all $t\in[0,T]$,
\begin{equation}
\label{eq:estimate_phi_g_and_w_hat}
\int_\Omega \varphi_g(T-t,x) \partial_t \hat{w}(t,x) + \nabla_x \varphi_g(T-t,x) \cdot \nabla_x \hat{w}(t,x) \ud x \leq 0.
\end{equation}
Finally, integrating \eqref{eq:estimate_phi_g_and_w_hat} with respect to time and using
\eqref{eq:identity_phi_g_and_w_hat}, we obtain
\begin{equation}
\label{eq:estimate_w_and_g}
\int_0^T \int_\Omega \hat{w}(t,x)g(t,x) \ud x \ud t \leq 0,
\end{equation}
for any nonnegative function $g \in C^\infty_c((0,T)\times\Omega;\R_+)$. It follows that
\begin{equation*}
    \hat{w}(t,x) \leq 0, \quad \text{a.e. } (t,x) \in [0,T]\times\Omega.
\end{equation*}
By \eqref{eq:definition_w_hat}, we then have
\begin{equation*}
    w(t,x) \leq 0, \quad \text{a.e. } (t,x) \in [0,T]\times\Omega.
\end{equation*}
Since, moreover, $w(t,x) = \inner{\mutx}{\left(y - y^*(t,x) \right)^2} \geq 0$, we conclude that
\begin{equation*}
    w(t,x) = 0, \quad \text{a.e. } (t,x) \in [0,T]\times\Omega.
\end{equation*}
\end{proof}
For clarity, let us rewrite equation \eqref{eq:w_equal_0} in terms of $\mutx$
\begin{equation}
\int_{\bm{Y}\times\bm{Z}} (y-y^*(t,x))^2 \ud \mutx(y,z) = 0, \quad \text{a.e. } (t,x)\in[0,T]\times\Omega.
\end{equation}
An immediate consequence to the equation above is that the marginal in $y$ of $\mutx$ is concentrated on $y^*(t,x)$. Stated differently, the measure $\mutx$ disintegrates as follows:
\begin{equation}
    \mutx(\mathrm{d}y,\mathrm{d}z_0,\mathrm{d}\overline{z}) = \delta_{y^*(t,x)}(\mathrm{d}y) \nu_{(t,x)}(\mathrm{d}z_0,\mathrm{d}\overline{z}) , \quad \text{a.e. } (t,x)\in[0,T]\times\Omega,
\end{equation}
where $\nu_{(t,x)}$ is the conditional probability measure on $z = (z_0,\overline{z})$, given $y$, and induced by $\mutx$. Therefore, for all continuous function $h\in C(\R)$, we have
\begin{equation}
    \inner{\mutx}{h(y)} = h(y^*(t,x)), \quad \text{a.e. } (t,x) \in [0,T]\times\Omega.
\end{equation}
We can use this concentration result and the fact that $w=0$ in \eqref{eq:PDE_w} to obtain for any $v\in H^1_0(\Omega)$ and for almost all $t\in[0,T]$,
\begin{equation}
\label{eq:intermediate_concentration_z_bar}
    \int_\Omega 2v(x) \inner{\mutx}{\left\lVert \overline{z} - \nabla_xy^*(t,x) \right\rVert^2} \ud x = 0.
\end{equation}
Since \eqref{eq:intermediate_concentration_z_bar} holds in particular for any $v\in C_c^\infty(\Omega)\subset H^1_0(\Omega)$, it follows that
\begin{equation}
    \inner{\mutx}{\left\lVert \overline{z} - \nabla_xy^*(t,x) \right\rVert^2} = 0, \quad \text{a.e. } (t,x) \in [0,T]\times\Omega.
\end{equation}
Therefore, $\mutx$ is concentrated on the graph of $\nabla_x y^*$ as well. We can update the disintegration of the measure $\mutx$ to obtain for almost all $(t,x)\in[0,T]\times\Omega$,
\begin{equation}
\mutx(\mathrm{d}y,\mathrm{d}z_0,\mathrm{d}\overline{z}) = \delta_{y^*(t,x)}(\mathrm{d}y) \delta_{\nabla_xy^*(t,x)}(\mathrm{d}\overline{z}) \theta_{(t,x)}(\mathrm{d}z_0),
\end{equation}
where $\theta_{(t,x)}$ is the conditional probability measure on $z_0$, given $y$ and $\overline{z}$.

Finally, we can use equation \eqref{eq:moment_dissipation_z0} satisfied by $\mutx$ and the concentration on $y^*$ and $\nabla_xy^*$ to obtain for any $\varphi \in C_c^\infty((0,T))$,
\begin{multline}
\label{eq:intermediate_equation_z0}
\int_0^T\int_\Omega \varphi(t)\inner{\mutx}{z_0^2} \ud x \ud t = \int_0^T \int_\Omega  \frac{1}{2}\varphi'(t)\left\lVert \nabla_x y^*(t,x) \right\rVert^2 \\
+ \varphi(t)f(t,x,y^*(t,x))\inner{\mutx}{z_0} \ud x \ud t.
\end{multline}
We know from Proposition \ref{prop:existence_of_weak_derivative_for_Young_measures} that
\begin{align}
\begin{split}
\label{eq:mutx_z0=partial_t_y}
    \inner{\mutx}{z_0} &= \partial_t \inner{\mutx}{y}=\partial_ty^*(t,x), \quad \text{a.e. } (t,x) \in [0,T]\times\Omega.
\end{split}
\end{align}
Using identity \eqref{eq:mutx_z0=partial_t_y} in \eqref{eq:intermediate_equation_z0} yields for any $\varphi \in C^\infty_c((0,T))$,
\begin{multline}
\label{eq:intermediate_equation_z0_and_y^*}
\int_0^T\int_\Omega \varphi(t)\inner{\mutx}{z_0^2} \ud x \ud t = \int_0^T \int_\Omega  \frac{1}{2}\varphi'(t)\left\lVert \nabla_x y^*(t,x) \right\rVert^2 \\
+ \varphi(t)\partial_t y^*(t,x) f(t,x,y^*(t,x)) \ud x \ud t.
\end{multline}
Since $y^*$ is a weak solution to the PDE \eqref{eq:semilinear_PDE}, it satisfies \eqref{eq:weak_formulation_dissipation}. Thus, \eqref{eq:intermediate_equation_z0_and_y^*} becomes for any $\varphi \in C^\infty_c((0,T))$,
\begin{align*}
\int_0^T\int_\Omega \varphi(t)\inner{\mutx}{z_0^2} \ud x \ud t = \int_0^T\int_\Omega \varphi(t) (\partial_t y^*(t,x))^2 \ud x \ud t,
\end{align*}
which writes equivalently
\begin{align}
\label{eq:equality_z0_partial_y}
\int_\Omega \inner{\mutx}{z_0^2} \ud x = \int_\Omega (\partial_t y^*(t,x))^2 \ud x, \quad \text{a.e. } t \in [0,T].
\end{align}
Let us note that identity \eqref{eq:mutx_z0=partial_t_y} implies for almost all $(t,x) \in [0,T]\times\Omega$,
\begin{equation}
\label{eq:intermediate_null_variance_z0}
    \inner{\mutx}{\left(z_0 - \partial_t y^*(t,x) \right)^2} = \inner{\mutx}{z_0^2} - (\partial_t y^*(t,x))^2.
\end{equation}
Therefore, combining \eqref{eq:equality_z0_partial_y} and \eqref{eq:intermediate_null_variance_z0} gives for almost all $t\in[0,T]$,
\begin{align*}
    \int_\Omega \inner{\mutx}{(z_0 - \partial_t y^*(t,x))^2} \ud x =  \int_\Omega \inner{\mutx}{z_0^2} - (\partial_t y^*(t,x))^2 \ud x = 0.
\end{align*}
Since $\inner{\mutx}{(z_0 - \partial_t y^*(t,x))^2} \geq 0$, it follows that
\begin{equation}
    \inner{\mutx}{(z_0 - \partial_t y^*(t,x))^2} = 0, \quad \text{a.e. } (t,x) \in [0,T]\times\Omega.
\end{equation}
We conclude that $\mutx$ is concentrated on the graph of $\partial_t y^*$ and writes for almost all $(t,x)\in[0,T]\times\Omega$,
\begin{equation}
\mutx(\mathrm{d}y,\mathrm{d}z_0,\mathrm{d}\overline{z}) = \delta_{y^*(t,x)}(\mathrm{d}y) \delta_{\nabla_xy^*(t,x)}(\mathrm{d}\overline{z}) \delta_{\partial_ty^*(t,x)}(\mathrm{d}z_0).
\end{equation}
This completes the proof of Theorem \ref{theorem:concentration_emv_solutions}. \qed

\section{Relation to occupation measure relaxations}\label{sec:occupation}

In this section, we clarify the link between the above developments and a previously available linear measure formulation designed for nonlinear PDEs and the moment-SOS hierarchy \cite{KHL-2018}. Note that in this reference, the focus was on the linear measure/moment formulation, and the relaxation gap issue was not addressed.

Let us define the set
\begin{equation*}
Q_T\coloneq [0,T]\times\overline{\Omega}    
\end{equation*}
and let $\bm{Y} \subset \R$ be a compact set defined as in \eqref{eq:compact_Y}. In this section, we assume that the time and space derivatives of the weak solution $y^*\in X_r$ to the PDE \eqref{eq:semilinear_PDE} are bounded. Therefore, there exists a compact set $\bm{Z}\subset \R\times\R^{n}$ such that
\begin{equation*}
    \left( \partial_t y^*(t,x), \nabla_xy^*(t,x) \right) \in \bm{Z}, \quad \text{a.e. } (t,x) \in [0,T]\times\Omega.
\end{equation*}
This additional assumption is made to simplify the exposition of the relaxation with occupation measures presented below. In particular, it allows us to work with measures supported on compact sets and avoids having to explicitly track the finite $r$-moments assumption that would otherwise be needed for measures with unbounded support.

Since $\Omega$ has a locally Lipschitz boundary, we may equip $\partial Q_T$ with the associated surface measure $\sigma$, so that the usual integration-by-parts formula in time and space holds. We denote by $$\eta(t,x) \coloneq \begin{bmatrix}
\eta_0(t,x) & \eta_1(t,x) & \ldots & \eta_n(t,x)
\end{bmatrix}^T$$ the outward unit surface normal vector to $\partial Q_T$. We also decompose the boundary of $Q_T$ as follows:
\begin{equation}
\label{eq:partition_boundary_QT}
    \partial Q_T = \partial Q_1 \cup \partial Q_2 \cup \partial Q_3,
\end{equation}
where
\begin{equation*}
    \partial Q_1 = \{0\}\times\overline{\Omega}, \qquad \partial Q_2 = \{T\}\times\overline{\Omega} \qquad \text{and} \qquad  \partial Q_3 = [0,T]\times\partial\Omega.
\end{equation*}
\begin{remark}
Note that any measure
$\hat{\mu}_\partial\in\mathcal{M}_+(\partial Q_T\times\bm{Y}\times\bm{Z})$
whose $(t,x)$-marginal coincides with the surface measure $\sigma$ on
$\partial Q_T$ can be decomposed according to the three boundary components as
\begin{equation*}
    \hat{\mu}_\partial = \hat{\mu}_{\partial,1} + \hat{\mu}_{\partial,2} + \hat{\mu}_{\partial,3},
\end{equation*}
where
$\hat{\mu}_{\partial,i}\in
\mathcal{M}_+(\partial Q_i\times\bm{Y}\times\bm{Z})$
is the restriction of $\hat{\mu}_\partial$ to
$\partial Q_i\times\bm{Y}\times\bm{Z}$, for each $i\in\{1,2,3\}$. This
decomposition is unique since the intersections of the boundary pieces
$\partial Q_i$ have zero surface measure, and hence zero
$\hat{\mu}_\partial$-measure.
\end{remark}
We next introduce the occupation-measure identities in the specific form needed
to compare the occupation-measure formulation with emv solutions. Except for
\eqref{eq:dissipation_occupation_measures}, these identities follow from the
admissibility conditions of the infinite-dimensional linear program
\cite[Equation~(16)]{KHL-2018} when applied to the PDE
\eqref{eq:semilinear_PDE}.
\begin{definition}[Occupation-measure identities]
\label{def:occupation_measure_identities}
The pair
$$
(\hat{\mu},\hat{\mu}_\partial)
\in
\mathcal{M}_+(Q_T\times\bm{Y}\times\bm{Z})
\times
\mathcal{M}_+(\partial Q_T\times\bm{Y}\times\bm{Z})
$$
is said to satisfy the \emph{occupation-measure identities} if for every test
function of the form
$$
\phi(t,x,y)=\varphi(t)v(x)\beta(y),
$$
with $\varphi\in C^\infty([0,T])$, $v\in H^1(\Omega)$ and
$\beta\in C^1(\bm{Y})$, the following identities hold:
\begin{subequations}
\begin{multline}
    \int_{Q_T\times\bm{Y}\times\bm{Z}} \left[ \frac{\partial \phi}{\partial t}(t,x,y) + z_0\frac{\partial \phi}{\partial y}(t,x,y) \right]\ud \hat{\mu}(t,x,y,z)\\
    - \int_{\partial Q_T\times\bm{Y}\times\bm{Z}} \phi(t,x,y) \eta_0(t,x) \ud \hat{\mu}_\partial(t,x,y,z) = 0, \label{eq:IBP_t_occupation_measure}
\end{multline}
\begin{multline}
    \int_{Q_T\times\bm{Y}\times\bm{Z}} \left[ \frac{\partial \phi}{\partial x_j}(t,x,y) + z_j\frac{\partial \phi}{\partial y}(t,x,y) \right] \ud \hat{\mu}(t,x,y,z)\\
    -\int_{\partial Q_T\times\bm{Y}\times\bm{Z}} \phi(t,x,y) \eta_j(t,x) \ud \hat{\mu}_\partial(t,x,y,z) = 0,\label{eq:IBP_x_occupation_measure}
\end{multline}
\text{for all $ j \in \{1,\,\ldots,\,n \}$ and}
\begin{multline}
\label{eq:weak_formulation_occupation_measure}
 \int_{Q_T\!\times\!\bm{Y}\!\times\!\bm{Z}} \!\Bigl[ \phi(t,x,y)\! \left( z_0 \!-\! f(t,x,y) \right) +\!\! \sum_{1\leq j \leq n} \!\!\Bigl( \frac{\partial \phi}{\partial x_j}(t,x,y) + z_j \frac{\partial \phi}{\partial y}(t,x,y)\Bigr)z_j \Bigr] \!\ud \hat{\mu}(t,x,y,z)\\
 - \sum_{1\leq j \leq n} \int_{\partial Q_T\times\bm{Y}\times\bm{Z}}  \phi(t,x,y) z_j \eta_j(t,x)  \ud \hat{\mu}_\partial(t,x,y,z) = 0,
\end{multline}
\begin{multline}
\label{eq:dissipation_occupation_measures}
\int_{Q_T\times\bm{Y}\times\bm{Z}} \varphi(t) z_0^2 - \frac{1}{2}\varphi'(t) \left\lVert \overline{z}\right\rVert^2 - \varphi(t) z_0 f(t,x,y) \ud \hat{\mu}(t,x,y,z) \\
+\int_{\partial Q_T\times\bm{Y}\times\bm{Z}} \frac{1}{2} \varphi(t) \left\lVert \overline{z} \right\rVert^2 \eta_0(t,x) \ud \hat{\mu}_\partial(t,x,y,z) = 0,
\end{multline}
\begin{equation}
\label{eq:initial_condition_occupation_measure}
\int_{\partial Q_1 \times \bm{Y}\times\bm{Z}} \phi(t,x,y) \ud \hat{\mu}_{\partial,1}(t,x,y,z) = \int_{\Omega} \phi(0,x,y_0(x)) \ud x,
\end{equation}
\begin{equation}
\label{eq:boundary_condition_occupation_measure}
\int_{\partial Q_3 \times \bm{Y}\times\bm{Z}} \phi(t,x,y) \ud \hat{\mu}_{\partial,3}(t,x,y,z) = \int_{[0,T]\times\partial\Omega} \phi(t,x,0) \ud \sigma(t,x),
\end{equation}
\text{and for all $\psi \in C^\infty(Q_T)$,}
\begin{equation}
\label{eq:normalization_occupation_measure}
\int_{\partial Q_i \times \bm{Y}\times\bm{Z}} \psi(t,x) \ud \hat{\mu}_{\partial,i}(t,x,y,z) = \int_{\partial Q_i} \psi(t,x) \ud \sigma(t,x), \quad \forall i \in \{1,2,3\}.
\end{equation}
\end{subequations}
\end{definition}
More precisely, constraints
\eqref{eq:IBP_t_occupation_measure} and \eqref{eq:IBP_x_occupation_measure}
encode respectively the integration-by-parts formulas in time and space for occupation
measures. Equation \eqref{eq:weak_formulation_occupation_measure}
encodes the occupation-measure version of the weak formulation of the PDE, whereas equation \eqref{eq:dissipation_occupation_measures} encodes the energy identity \eqref{eq:weak_formulation_dissipation}. The constraints \eqref{eq:initial_condition_occupation_measure} and
\eqref{eq:boundary_condition_occupation_measure} impose the initial condition at $t=0$ and the homogeneous Dirichlet boundary condition on $[0,T]\times\partial\Omega$, respectively. Finally, \eqref{eq:normalization_occupation_measure} is a normalization
constraint requiring the $(t,x)$-marginal of $\hat{\mu}_{\partial,i}$ to coincide with the surface measure $\sigma$ on $\partial Q_i$.

\begin{remark}
\label{remark:disintegration_occupation_measure}
In Definition \ref{def:occupation_measure_identities}, $(\hat{\mu},\hat{\mu}_\partial)$ satisfies the integration-by-parts formula for occupation measures \eqref{eq:IBP_t_occupation_measure} and \eqref{eq:IBP_x_occupation_measure}, together with the normalization constraints \eqref{eq:normalization_occupation_measure} on $\hat{\mu}_\partial$. Therefore, the $(t,x)$-marginal of $\hat{\mu}$ is the Lebesgue measure on $Q_T$. By the disintegration theorem, there exists a family of probability measures $\{\mutx\}_{(t,x)\in Q_T}$ on
$\bm{Y}\times\bm{Z}$, defined for almost every $(t,x)\in Q_T$, such that
\begin{equation}
\label{eq:disintegration_occupation_measure}
\ud\hat{\mu}(t,x,y,z) = \ud\mutx(y,z) \ud x \ud t.
\end{equation}
In particular, the map $\mu\colon (t,x)\mapsto \mutx$ defines a Young measure on $\bm{Y}\times\bm{Z}$ parametrized by $Q_T$.
\end{remark}
In the next proposition, we show how the moment constraints for emv solutions introduced in this work can be recovered from the occupation-measure identities satisfied by $(\hat{\mu},\hat{\mu}_\partial)$.
\begin{proposition}
\label{prop:link_KHL_and_emv}
Suppose that the pair $(\hat{\mu},\hat{\mu}_\partial)\in \mathcal{M}_+(Q_T\times\bm{Y}\times\bm{Z})\times\mathcal{M}_+(\partial Q_T\times\bm{Y}\times\bm{Z})$ satisfies the occupation-measure identities. Then the Young measure $\mu$ obtained from the disintegration \eqref{eq:disintegration_occupation_measure} of the occupation measure $\hat{\mu}$ is an emv solution to the PDE \eqref{eq:semilinear_PDE}. 
\end{proposition}
\begin{proof}
Let $(\hat{\mu},\hat{\mu}_\partial)\in \mathcal{M}_+(Q_T\times\bm{Y}\times\bm{Z})\times\mathcal{M}_+(\partial Q_T\times\bm{Y}\times\bm{Z})$ be a pair of measures that satisfies the occupation-measure identities. By Remark \ref{remark:disintegration_occupation_measure}, the disintegration \eqref{eq:disintegration_occupation_measure} holds for some Young measure $\mu \in \mathcal{Y}(Q_T;\bm{Y}\times\bm{Z})$. Since both $\bm{Y}$ and $\bm{Z}$ are compact, it follows that $\mu$ has finite $p$-moments for any $p \in [1,\infty)$.

First, the integration-by-parts formulas \eqref{eq:IBP_in_time} and
\eqref{eq:IBP_in_space} for the Young measure $\mu$ are recovered by choosing test functions of the form $\phi(t,x,y)=\varphi(t)v(x)\beta(y)$ in
\eqref{eq:IBP_t_occupation_measure} and \eqref{eq:IBP_x_occupation_measure}, with $\varphi\in C_c^\infty((0,T))$, $v\in C_c^\infty(\Omega)$ and $\beta\in C^1(\bm{Y})$.

Then, we recover the weak formulation \eqref{eq:weak_formulation_m1} satisfied by
$m_1$ by choosing $\phi(t,x,y)=\varphi(t)v(x)$ in \eqref{eq:weak_formulation_occupation_measure}, where $\varphi\in C_c^\infty((0,T))$ and $v\in H^1_0(\Omega)$. With this choice,
\eqref{eq:weak_formulation_occupation_measure} becomes
\begin{align*}
0 &= \int_{Q_T\times\bm{Y}\times\bm{Z}} \Bigl[ \varphi(t) v(x) \left( z_0 - f(t,x,y) \right) +\! \sum_{1\leq j \leq n} \!\!\Bigl( \varphi(t) \frac{\partial v}{\partial x_j}(t,x,y) z_j\Bigr) \Bigr] \ud \hat{\mu}(t,x,y,z)\\
&= \int_{Q_T\times\bm{Y}\times\bm{Z}}  \varphi(t) \Bigl[ v(x) z_0 + \nabla_x v(x) \cdot \overline{z} - v(x) f(t,x,y) \Bigl] \ud \hat{\mu}(t,x,y,z)\\
&= \int_0^T \varphi(t) \biggl[ \int_{\Omega} v(x) \inner{\mutx}{z_0} + \nabla_xv(x) \cdot \inner{\mutx}{\overline{z}} - v(x) \inner{\mutx}{f(t,x,y)} \ud x \biggr] \ud t.
\end{align*}
Using the fact that $\mu$ satisfies the integration-by-parts formulas for Young measures \eqref{eq:IBP_in_time} and \eqref{eq:IBP_in_space}, we obtain $\inner{\mutx}{z_0} = \partial_t m_1(t,x)$ and $\inner{\mutx}{\overline{z}} = \nabla_x m_1(t,x)$ from Proposition \ref{prop:existence_of_weak_derivative_for_Young_measures}. Therefore, it follows from the above equation that
\begin{equation*}
\int_0^T \varphi(t) \biggl[ \int_{\Omega} v(x) \partial_t m_1(t,x) + \nabla_xv(x) \cdot \nabla_x m_1(t,x) - v(x) m_f(t,x) \ud x \biggr] \ud t = 0,
\end{equation*}
which coincides with \eqref{eq:weak_formulation_m1}. Similarly, we recover the weak formulation \eqref{eq:weak_formulation_m2} satisfied by $m_2$ by choosing $\phi(t,x,y) = \varphi(t) v(x) y$ in \eqref{eq:weak_formulation_occupation_measure}, where $\varphi \in C^\infty_c((0,T))$ and $v \in H^1_0(\Omega)$.

The energy identity \eqref{eq:moment_dissipation_z0} follows directly from \eqref{eq:dissipation_occupation_measures} evaluated with $\varphi \in C^\infty_c((0,T))$, together with the disintegration \eqref{eq:disintegration_occupation_measure} of the occupation measure $\hat{\mu}$.

Next, we proceed to prove that $m_1$ and $m_2$ satisfy respectively the initial conditions \eqref{eq:initial_condition_m1} and \eqref{eq:initial_condition_m2}. Let $v \in C^\infty_c(\Omega)$ and $\varphi \in C_c^\infty([0,T))$. Suppose in addition that $\varphi(0)=1$. Then,
\begin{align*}
    \int_\Omega v(x) m_1(0,x) \ud x &= \int_\Omega \varphi(0) v(x) m_1(0,x) \ud x - \int_\Omega \varphi(T) v(x) m_1(T,x) \ud x \\
    &=-\int_0^T \frac{\!\ud}{\!\ud t} \left( \int_\Omega \varphi(t)v(x)m_1(t,x) \ud x \right) \ud t.
\end{align*}
Then,
\begin{align*}
    \int_\Omega v(x) m_1(0,x) \ud x &=-\int_0^T  \int_\Omega \varphi'(t)v(x) m_1(t,x) + \varphi(t)v(x) \partial_t m_1(t,x) \ud x \ud t \\
    &= -\int_0^T  \int_\Omega \int_{\bm{Y}\times\bm{Z}} \varphi'(t)v(x)y + \varphi(t)v(x) z_0 \ud \mutx(y,z) \ud x \ud t\\
    &= -\int_{Q_T\times\bm{Y}\times\bm{Z}} \varphi'(t)v(x)y + \varphi(t)v(x) z_0 \ud \hat{\mu}(t,x,y,z).
\end{align*}
Therefore, letting $\phi(t,x,y) = \varphi(t)v(x)y$ and using the integration-by-parts formula in time for occupation measures \eqref{eq:IBP_t_occupation_measure} in the equation above yields
\begin{align*}
    \int_\Omega v(x) m_1(0,x) \ud x &= -\int_{Q_T\times\bm{Y}\times\bm{Z}} \frac{\partial \phi}{\partial t}(t,x,y) + \frac{\partial \phi}{\partial y}(t,x,y) z_0 \ud \hat{\mu}(t,x,y,z)\\
    &= \int_{\partial Q_1 \!\times\! \bm{Y}\!\times\!\bm{Z}} \!\!\phi(t,x,y) \ud \hat{\mu}_{\partial,1}(t,x,y,z) -\!\! \int_{\partial Q_2 \!\times\! \bm{Y}\!\times\!\bm{Z}} \!\!\phi(t,x,y) \ud \hat{\mu}_{\partial,2}(t,x,y,z).
\end{align*}
Finally, since $\hat{\mu}_{\partial,1}$ satisfies \eqref{eq:initial_condition_occupation_measure} and since $\phi = 0$ on $\partial Q_2\times\bm{Y}\times\bm{Z}$, it follows that
\begin{align*}
    \int_\Omega v(x) m_1(0,x) \ud x = \int_{\Omega} \phi(0,x,y_0(x)) \ud x=\int_{\Omega} v(x)y_0(x) \ud x.
\end{align*}
Therefore, $m_1$ satisfies the initial condition
\begin{equation*}
    m_1(0,x) = y_0(x), \quad \text{for a.e. } x\in\Omega.
\end{equation*}
Similarly, one can recover the initial condition \eqref{eq:initial_condition_m2} on $m_2$ by considering the test function $\phi(t,x,y) = \varphi(t)v(x)y^2$ in \eqref{eq:IBP_t_occupation_measure}.

Finally, the homogeneous Dirichlet boundary conditions on $m_1$ and $m_2$ are obtained with the same method.
\end{proof}

\section{Conclusion}\label{sec:conclusion}
This paper introduces the notion of energy measure-valued solutions for semilinear parabolic PDEs. Such solutions are parametrized Young measures whose first- and second-order moments, involving both the state variable and its derivatives, satisfy variational equalities derived from the original PDE. The main result of the paper is that such solutions always concentrate on the classical weak solution of the PDE under mild assumptions on the nonlinear term of the equation. This constitutes a first step toward proving a result of absence of relaxation gap in the more involved context of PDE-constrained optimization. This could be done by showing that there is no relaxation gap when optimizing over a set of parametrized measured-valued solutions of a parabolic PDE, as what has been done in \cite{cardoen2024} for the hyperbolic case. Moreover, in the framework of optimal control of parabolic PDEs, one could introduce a control variable in the Young measures as in  \cite{lebarbe2025optimal} and  show some superposition principle to get the absence of relaxation gap. Another possible extension of the proposed work could be to show concentration of measure-valued solutions of semilinear PDE systems that depend nonlinearly on the solution and its spatial derivatives. This class of PDEs is relevant in the literature since it includes some formulations of the Navier-Stokes equations.

\section*{Acknowledgements}
The authors acknowledge the use of AI for assistance with brainstorming ideas and language editing. The final content, analysis and conclusions remain the sole responsibility of the authors.

This work was funded by the European Union under the project ROBOPROX (reg.~no.~CZ.02.01.01 / 00 / 22\_008/0004590) and by the ANR-DFG project MONET (ANR-25-CE48-6598-01). This work was also supported by the AID (Agence de l’Innovation de D\'{e}fense) from the French Ministry of the Armed Forces (Minist\`{e}re des Arm\'{e}es).

\begin{appendices}

\section{Proof of Proposition \ref{prop:properties_sobolev}}
\label{appendix:proof_properties_sobolev}

We first prove Proposition \ref{prop:properties_sobolev_W1r}.

Let $(u,v) \in W^{1,r}(\Omega)\times H^1(\Omega)$. Since $u \in W^{1,r}(\Omega) \subset C(\overline{\Omega})$ is bounded, one has immediately $uv \in L^2(\Omega)$. Now, let us consider
\begin{equation*}
    g \coloneq u \nabla_x v + v \nabla_x u.
\end{equation*}
The goal is to show that $g \in (L^2(\Omega))^n$. Since $\nabla_x v \in (L^2(\Omega))^n$ and since $u$ is bounded, it follows that $u \nabla_x v \in (L^2(\Omega))^n$. It remains to show that $v \nabla_x u \in (L^2(\Omega))^n$.

\noindent First, consider the case where $n>2$. The Sobolev embedding theorem (see \cite[Theorem 4.12]{adams_sobolev_2003}) yields the continuous embedding
\begin{equation}
H^1(\Omega)\hookrightarrow L^{p^*}(\Omega) \quad \text{with} \quad p^* = \frac{2n}{n-2}. \label{eq:embedding_subcrit}
\end{equation}
Since $v\in H^1(\Omega) \hookrightarrow L^{p^*}(\Omega)$ and $\nabla_x u \in (L^r(\Omega))^n$, it follows from Hölder’s inequality that $v \nabla_x u \in (L^s(\Omega))^n$, where $s$ is defined by
\begin{align*}
    \frac{1}{s} = \frac{1}{r} + \frac{1}{p^*}= \frac{1}{r} - \frac{1}{n} + \frac{1}{2}.
\end{align*}
Since $r$ is supposed to satisfy $r > n$, it follows that $s \geq 2$. Therefore, $v \nabla_x u \in (L^2(\Omega))^n$.

\noindent In the case $n=2$, the Sobolev embedding theorem yields
\begin{equation}
H^1(\Omega)\hookrightarrow L^{q}(\Omega), \quad \forall q \in [2,\infty). \label{eq:embedding_crit}
\end{equation}
It follows from Hölder’s inequality that $v\nabla_x u \in (L^{qr/(q+r)}(\Omega))^n$ for all $q \in [2,\infty)$. Since $r>n=2$, one can take $q = 2r/(r-2)$ and immediately obtain that $v\nabla_x u \in (L^{2}(\Omega))^n$.

\noindent Finally, in the case $n=1$, the Sobolev embedding theorem yields
\begin{equation}
H^1(\Omega)\hookrightarrow L^{\infty}(\Omega). \label{eq:embedding_supcrit}
\end{equation}
Therefore $v$ is bounded, and since $\nabla_x u \in (L^r(\Omega))^n \subset (L^2(\Omega))^n$, it follows that $v\nabla_x u \in (L^2(\Omega))^n$.

\noindent In any case, we have proven that $\nabla_x(uv) = g \in (L^2(\Omega))^n$. Thus, $uv \in H^1(\Omega)$.

We now proceed to prove Proposition \ref{prop:properties_sobolev_Lr2}.

\noindent Let $1/2 \leq \alpha \leq 1$ and let $(u,v) \in L^{\alpha r}(\Omega)\times H^1(\Omega)$. Let us first consider the case $n>2$. By the embedding \eqref{eq:embedding_subcrit} and by Hölder’s inequality, it follows that $uv \in L^s(\Omega)$, where
\begin{align*}
    \frac{1}{s} = \frac{1}{\alpha r} + \frac{1}{2n/(n-2)} = \frac{1-\alpha}{\alpha r} + \frac{1}{r} - \frac{1}{n} + \frac{1}{2}.
\end{align*}
Since $r>n$, it follows that
\begin{align*}
    \frac{1}{s} \leq \frac{1-\alpha}{\alpha r} + \frac{1}{2} = \frac{2(1-\alpha) + \alpha r}{2\alpha r} = \frac{1}{2\alpha}\left( \alpha \left(1 - \frac{2}{r} \right) + \frac{2}{r} \right).
\end{align*}
Since $r\geq 2$, the function $\alpha \mapsto \alpha \left(1 - \frac{2}{r} \right) + \frac{2}{r}$ is nondecreasing on $[1/2,1]$ and attains its maximum at $\alpha=1$. Therefore,
\begin{equation*}
    \frac{1}{s} \leq \frac{1}{2\alpha}.
\end{equation*}
Thus, $s \geq 2\alpha$ and therefore $uv \in L^{2\alpha}(\Omega)$.

\noindent In the case $n=2$, the embedding \eqref{eq:embedding_crit} and Hölder’s inequality imply that $uv \in L^{q\alpha r/(q+\alpha r)}(\Omega)$ for all $q\in[2,\infty)$. Since $r > n = 2$, one can take $q = 2\alpha r/(r-2)$ and immediately obtain that $uv \in L^{2\alpha}(\Omega)$.

\noindent Finally, the case $n=1$ follows directly from the embedding \eqref{eq:embedding_supcrit} and the fact that $u \in L^{\alpha r}(\Omega) \subset L^{2\alpha}(\Omega)$. \qed

\section{Proof of Proposition \ref{prop:uniqueness_weak_solution}}
\label{appendix:proof_uniqueness_weak_solutions}
Suppose that $y_1, y_2 \in X_r\subset C([0,T]\times\overline{\Omega})$ are two weak solutions to the PDE. By Definition \ref{def:weak_solutions}, they satisfy for all $v\in H^1_0(\Omega)$ and almost all $t\in[0,T]$,
\begin{multline}
\label{eq:proof_unique_weak_form}
\int_\Omega v(x) \partial_t(y_1-y_2)(t,x) + \nabla_x v(x) \cdot \nabla_x (y_1-y_2)(t,x)\\
    - v(x) (f(t,x,y_1(t,x))-f(t,x,y_2(t,x))) \ud x=0.
\end{multline}
Notice that for all $t\in[0,T]$, one has $v_t=y_1(t,\cdot)-y_2(t,\cdot) \in W^{1,r}_0(\Omega)\subset H^1_0(\Omega)$. Therefore, using $v_t$ as a test function in \eqref{eq:proof_unique_weak_form} yields for almost all $t\in[0,T]$,
\begin{multline}
\label{eq:proof_unique_vt}
\int_\Omega (y_1(t,x)-y_2(t,x)) \partial_t(y_1-y_2)(t,x) \ud x = - \int_\Omega \left\lVert \nabla_x (y_1 - y_2)(t,x) \right\rVert^2 \ud x\\ +\int_\Omega (y_1(t,x)-y_2(t,x)) \left(f(t,x,y_1(t,x))-f(t,x,y_2(t,x))\right) \ud x.
\end{multline}
Since $y_1, y_2 \in X_r \subset C([0,T]\times\overline{\Omega})$, both $y_1$ and $y_2$ are bounded. Therefore, there exists a compact set $K$ such that for any $t\in[0,T]$ and $x\in\overline{\Omega}$, $y_1(t,x)\in K$ and $y_2(t,x)\in K$. By Assumption \ref{assump:OSL_f} on $f$, it follows that there exists $L_K\geq0$ such that for almost all $t\in[0,T]$,
\begin{equation}
\label{eq:proof_unique_OSL}
\int_\Omega \!(y_1(t,x)-y_2(t,x)) \!\left(f(t,x,y_1(t,x))\!-\!f(t,x,y_2(t,x))\right) \!\ud x \leq \!L_K\! \left\lVert y_1(t,\cdot) - y_2(t,\cdot) \right\rVert^2_{L^2(\Omega)}\!\!.
\end{equation}
Moreover, by the Poincaré inequality, there exists a constant $C(\Omega)\geq0$ such that for any $t\in[0,T]$,
\begin{equation}
\label{eq:proof_unique_poincare}
    C(\Omega) \left\lVert y_1(t,\cdot) - y_2(t,\cdot) \right\rVert^2_{L^2(\Omega)} \leq \int_\Omega \left\lVert \nabla_x (y_1 - y_2)(t,x) \right\rVert^2 \ud x.
\end{equation}
Applying the estimates \eqref{eq:proof_unique_OSL} and \eqref{eq:proof_unique_poincare} in \eqref{eq:proof_unique_vt} gives for almost all $t\in[0,T]$,
\begin{equation}
\label{eq:proof_unique_ineq_intermediate}
\int_\Omega (y_1(t,x)-y_2(t,x)) \partial_t(y_1-y_2)(t,x) \ud x \leq (L_K-C(\Omega)) \left\lVert y_1(t,\cdot) -  y_2(t,\cdot) \right\rVert^2_{L^2(\Omega)}.
\end{equation}
Since $y_1-y_2 \in X_r$, it follows from \cite[Theorem 3, 5.9.2]{evans_pde} that the mapping $t\mapsto \left\lVert y_1(t,\cdot) -  y_2(t,\cdot) \right\rVert^2_{L^2(\Omega)}$ is absolutely continuous, with
\begin{equation*}
    \int_\Omega (y_1(t,x)-y_2(t,x)) \partial_t(y_1-y_2)(t,x) \ud x = \frac{1}{2}\frac{\ud }{\ud t} \left\lVert y_1(t,\cdot) -  y_2(t,\cdot) \right\rVert^2_{L^2(\Omega)}
\end{equation*}
for almost all $t\in[0,T]$. Therefore, equation \eqref{eq:proof_unique_ineq_intermediate} becomes for almost all $t\in[0,T]$,
\begin{equation*}
\frac{\ud }{\ud t} \left\lVert y_1(t,\cdot) -  y_2(t,\cdot) \right\rVert^2_{L^2(\Omega)} \leq 2(L_K-C(\Omega)) \left\lVert y_1(t,\cdot) -  y_2(t,\cdot) \right\rVert^2_{L^2(\Omega)}.
\end{equation*}
It follows from Grönwall's inequality that for every $t\in[0,T]$,
\begin{equation*}
    \left\lVert y_1(t,\cdot) -  y_2(t,\cdot) \right\rVert^2_{L^2(\Omega)} \leq e^{2t\left(L_K-C(\Omega)\right)} \left\lVert y_1(0,\cdot) -  y_2(0,\cdot) \right\rVert^2_{L^2(\Omega)}.
\end{equation*}
Since $y_1$ and $y_2$ satisfy the same initial condition \eqref{eq:initial_condition_weak_solution}, one eventually obtains for every $t\in[0,T]$,
\begin{equation*}
\left\lVert y_1(t,\cdot) -  y_2(t,\cdot) \right\rVert^2_{L^2(\Omega)}= 0.
\end{equation*}
Therefore, $y_1=y_2$ on $[0,T]\times\overline{\Omega}$.\qed

\section{Proof of Proposition \ref{prop:energy_identity_for_weak_solutions}}
\label{appendix:proof_energy_identity_for_weak_solutions}

Suppose that $y^* \in X_r \subset C([0,T]\times\overline{\Omega})$ is a weak solution to the PDE \eqref{eq:semilinear_PDE}.

Let us derive identity \eqref{eq:weak_formulation_energy} first. For $v \in H^1_0(\Omega)$ and $t \in [0,T]$, define $\tilde v_t := v\,y^*(t,\cdot)$. Because $y^*(t,\cdot)\in W^{1,r}_0(\Omega)$ for any $t\in[0,T]$, Proposition \ref{prop:properties_sobolev_W1r} shows that $\tilde v_t \in H^1_0(\Omega)$, with the product rule $\nabla_x \tilde{v}_t = y^*(t,\cdot) \nabla_x v + v \nabla_x y^*(t,\cdot)$. Therefore, using $\tilde{v}_t$ as a test function in \eqref{eq:weak_solution_ae_t} gives for almost all $t\in[0,T]$,
\begin{multline}
\label{eq:proof_energy_weak_vt}
    \int_\Omega v(x)y^*(t,x) \partial_ty^*(t,x) + \left(y^*(t,x)\nabla_x v(x) + v(x)\nabla_x y^*(t,x) \right)\cdot \nabla_x y^*(t,x)\\
    - v(x)y^*(t,x) f(t,x,y^*(t,x)) \ud x=0.
\end{multline}
Since $y^* \in X_r$, one has $(y^*)^2 \in X_r$ by Proposition \ref{prop:product_in_Xr}, where the space and time derivatives of $(y^*)^2$ are given by \eqref{eq:product_rule_space_Xr} and \eqref{eq:product_rule_time_Xr}, respectively. Using these two identities in \eqref{eq:proof_energy_weak_vt} gives
\begin{multline*}
    \int_\Omega v(x)\partial_t((y^*)^2/2)(t,x) + \nabla_x v(x) \cdot \nabla_x((y^*)^2/2)(t,x)- v(x)y^*(t,x) f(t,x,y^*(t,x))\\
    + v(x) \left\lVert \nabla_xy^*(t,x) \right\rVert^2 \!\ud x = 0.
\end{multline*}
We obtain \eqref{eq:weak_formulation_energy} by multiplying the previous equation by $2\varphi$ for any $\varphi \in C_c^\infty((0,T))$, and integrating in time.

Now, we proceed to derive identity \eqref{eq:weak_formulation_dissipation}. Notice that for almost all $t\in[0,T]$, one has $v_t =  \partial_t y^*(t,\cdot) \in W^{1,r}_0(\Omega) \subset H^1_0(\Omega)$. Therefore, using $v_t$ as a test function in \eqref{eq:weak_solution_ae_t} yields for almost all $t\in[0,T]$,
\begin{equation}
\label{eq:proof_energy_weak_dissipation}
\int_\Omega (\partial_t y^*(t,x))^2 \ud x = \int_\Omega -\nabla_x (\partial_t y^*)(t,x)\cdot\nabla_x y^*(t,x) +\partial_t y^*(t,x) f(t,x,y^*(t,x)) \ud x.
\end{equation}
The next step of the proof is to justify that one can swap the time and space derivatives $\nabla_x(\partial_t y^*) = \partial_t(\nabla_x y^*)$. First, since $y^* \in X_r=H^1(0,T;W^{1,r}_0(\Omega))$, one has from \cite[Theorem 2, 5.9.2]{evans_pde} that for any $t\in[0,T]$,
\begin{equation}
\label{eq:proof_energy_weak_absolut_continu}
    y^*(t,\cdot) = y^*(0,\cdot) + \int_0^t \partial_sy^*(s,\cdot) \ud s.
\end{equation}
Then, using the fact that $\nabla_x$ defines a bounded linear operator from $W^{1,r}_0(\Omega)$ to $(L^r(\Omega))^n$, and the fact that the Bochner integral commutes with bounded linear operators (see for instance \cite[Section 1.2.a]{analysis_in_Banach_space}), taking the space gradient of \eqref{eq:proof_energy_weak_absolut_continu} gives for any $t\in[0,T]$,
\begin{equation}
\label{eq:proof_energy_weak_absolut_continu_gradient}
    \nabla_x y^*(t,\cdot) = \nabla_xy^*(0,\cdot) + \int_0^t \nabla_x(\partial_s y^*(s,\cdot)) \ud s.
\end{equation}
Therefore, $\nabla_x y^* \in H^1(0,T;(L^r(\Omega))^n)$ and one can directly identify from \eqref{eq:proof_energy_weak_absolut_continu_gradient} its weak derivative in time as
\begin{equation}
    \partial_t(\nabla_xy^*) = \nabla_x(\partial_t y^*).
\end{equation}
It follows from \eqref{eq:proof_energy_weak_dissipation} that for almost all $t\in[0,T]$,
\begin{equation}
\label{eq:proof_energy_weak_dissipation_intermediate}
\int_\Omega (\partial_t y^*(t,x))^2 \ud x = \int_{\Omega} - \partial_t (\nabla_x y^*)(t,x) \cdot \nabla_x y^*(t,x) + \partial_t y^*(t,x) f(t,x,y^*(t,x)) \ud x.
\end{equation}
Let us consider the functional
\begin{equation*}
E(t) \coloneq \int_\Omega \frac{1}{2}\left\lVert \nabla_x y^*(t,x) \right\rVert^2 \ud x, \quad \forall t \in [0,T].
\end{equation*}
Applying Lions-Magenes lemma (see for instance \cite[Lemma 1.2, Chapter III]{temam_Lions_magenes_lemma} for an explicit statement of the lemma) to $\nabla_x y^* \in H^1(0,T;(L^r(\Omega))^n) \subset H^1(0,T;(L^2(\Omega))^n)$ gives for almost all $t\in[0,T]$,
\begin{equation}
\label{eq:proof_energy_weak_derivative_E}
E'(t) = \frac{1}{2}\frac{\!\ud}{\!\ud t} \left\lVert \nabla_x y^*(t,\cdot) \right\rVert_{(L^2(\Omega))^n}^2 = \int_\Omega \partial_t(\nabla_xy^*)(t,x) \cdot \nabla_x y^*(t,x) \ud x.
\end{equation}
Therefore, injecting \eqref{eq:proof_energy_weak_derivative_E} in \eqref{eq:proof_energy_weak_dissipation_intermediate} yields for almost all $t\in[0,T]$,
\begin{equation}
\label{eq:proof_energy_weak_intermediate_E}
\int_\Omega (\partial_t y^*(t,x))^2 \ud x = -E'(t)  + \int_{\Omega}  \partial_t y^*(t,x) f(t,x,y^*(t,x)) \ud x.
\end{equation}
Multiplying \eqref{eq:proof_energy_weak_intermediate_E} by any $\varphi \in C_c^\infty((0,T))$ and integrating in time gives
\begin{align*}
\int_0^T  \!\!\!\! \int_\Omega \varphi(t) (\partial_t y^*(t,x))^2 \!\ud x\! \ud t &= \!\!\!\int_0^T \!\!\!\!\!\!-\varphi(t)E'(t) \!\ud t +\! \int_0^T \!\!\!\!\int_\Omega \varphi(t) \partial_ty^*(t,x)f(t,x,y^*(t,x)) \!\ud x \!\ud t.
\end{align*}
After integrating by parts in time in the equation above, one finally obtains
\begin{align*}
\int_0^T  \!\!\!\! \int_\Omega \varphi(t) (\partial_t y^*(t,x))^2 \!\ud x\! \ud t &= \!\!\!\int_0^T \!\!\!\varphi'(t)E(t) \!\ud t + \int_0^T \!\!\!\!\int_\Omega \varphi(t) \partial_ty^*(t,x)f(t,x,y^*(t,x)) \!\ud x \!\ud t,
\end{align*}
which is exactly equation \eqref{eq:weak_formulation_dissipation}.\qed

\end{appendices}

\end{document}